\documentclass[a4paper,12pt]{article}
\usepackage[cp1251]{inputenc}
\usepackage[russian]{babel}
\usepackage{amsfonts, amssymb, amsmath, amsthm, amscd}
\usepackage{cite}
\ifx\undefined \pdfgentounicode \else
\input{glyphtounicode} \pdfgentounicode=1
\fi

\author{A.A. Vasil'eva}
\title{Widths of weighted Sobolev classes with restrictions $f(a)=\dots=f^{(k-1)}(a)=f^{(k)}(b)=\dots
=f^{(r-1)}(b)=0$ and spectra of non-linear differential equations}
\date{}
\begin{document}

\maketitle

\newenvironment{Biblio}{%
                  \renewcommand{\refname}{\footnotesize REFERENCES}%
                  \begin{thebibliography}{99}%
                  \itemsep -0.2em}%
                  {\end{thebibliography}}

\def\inff{\mathop{\smash\inf\vphantom\sup}}
\renewcommand{\le}{\leqslant}
\renewcommand{\ge}{\geqslant}
\newcommand{\sgn}{\mathrm {sgn}\,}
\newcommand{\inter}{\mathrm {int}\,}
\newcommand{\dist}{\mathrm {dist}}
\newcommand{\supp}{\mathrm {supp}\,}
\newcommand{\R}{\mathbb{R}}
\renewcommand{\C}{\mathbb{C}}
\newcommand{\Z}{\mathbb{Z}}
\newcommand{\N}{\mathbb{N}}
\newcommand{\Q}{\mathbb{Q}}
\theoremstyle{plain}
\newtheorem{Trm}{Theorem}
\newtheorem{trma}{Theorem}

\newtheorem{Def}{Definition}
\newtheorem{Cor}{Corollary}
\newtheorem{Lem}{Lemma}
\newtheorem{Rem}{Remark}
\newtheorem{Sta}{Proposition}
\newtheorem{Exa}{Example}
\renewcommand{\proofname}{\bf Proof}
\renewcommand{\thetrma}{\Alph{trma}}

\section{Introduction}

In papers of Pinkus \cite{pinkus_constr}, Buslaev and Tikhomirov
\cite{busl_tikh} it was proved that the Kolmogorov widths of the
Sobolev class $W^r_p[a, \, b]$ with some boundary conditions in
the space $L_q[a, \, b]$ for $p\ge q$ coincide with spectral
numbers of some special non-linear differential equation with the
same boundary conditions. For weighted Sobolev classes with
piecewise continuous weights such theorems were obtained by
Buslaev \cite{busl_dan}, \cite{busl_aa}, and for $r=1$ and $g\in
L_{p'}[a, \, b]$, $v\in L_q[a, \, b]$, by Edmunds and Lang
\cite{edm_lang_2008}.

Buslaev and Tikhomirov \cite{busl_tikh} considered the following
boundary conditions for functions $x$: 1) $x^{(j)}(a)=0$, $0\le
j\le r-1$, 2) $x^{(j)}(a) =0$, $x^{(j)}(b)=0$, $0\le j\le r-1$, 3)
no boundary conditions, 4) $x^{(j)}(a)=0$ for even $j$,
$x^{(j)}(b)=0$ for odd $j$, $0\le j\le r-1$, 5) periodic
conditions. In case 1) this result was generalized in
\cite{vas_spn} for weighted Sobolev classes in weighted Lebesgue
space. It was supposed that the corresponding two-weighted Riemann
-- Liouville operator (see formula (\ref{tirmu1})) is compact.
Moreover, it was proved that the widths always strictly decrease.
Here we obtain the analogue of this theorem for the boundary
conditions $x(a)=\dots =x^{(k-1)}(a)=x^{(k)}(b)=\dots
=x^{(r-1)}(b)=0$.

We give the definition of the Kolmogorov, linear, Gelfand and
Bernstein widths.

Let $(X, \, \|\cdot\|_X)$ be a normed space, let $X^*$ be its
dual, and let ${\cal L}_n(X)$, $n\in \Z_+$, be the family of
subspaces of $X$ of dimension at most $n$. Denote by $L(X, \, Y)$
the space of continuous linear operators from $X$ into a normed
space $Y$. Also, by ${\rm rk}\, A$ denote the dimension of the
image of an operator $A\in L(X, \, Y)$, and by $\| A\|
_{X\rightarrow Y}$, its norm.

By the Kolmogorov $n$-width of an absolutely convex set $M\subset
X$ in the space $X$, we mean the quantity
$$
d_n(M, \, X)=\inff _{L\in {\cal L}_n(X)} \sup_{x\in M}\inff_{y\in
L}\|x-y\|_X,
$$
by the linear $n$-width, the quantity
$$
\lambda_n(M, \, X) =\inff_{A\in L(X, \, X), \, {\rm rk} A\le
n}\,\sup _{x\in M}\| x-Ax\| _X,
$$
by the Gelfand $n$-width, the quantity
$$
d^n(M, \, X)=\inff _{x_1^*, \, \dots, \, x_n^*\in X^*} \sup
\{\|x\|:\; x\in M, \, x^*_j(x)=0, \; 1\le j\le n\},
$$
and by the Bernstein $n$-width, the quantity
$$
b_n(M, \, X) = \sup _{\dim L=n+1} \sup \{R>0:\; B_R(0)\cap
L\subset M\}
$$
(here the supremum is taken over the family of linear
$n+1$-dimensional subspaces). The following relations between
widths are well-known (here $M$ is an algebraic sum of a compact
and a finite dimensional subspace; see \cite[p. 207]{itogi_nt}):
\begin{align}
\label{bndnln} b_n(M, \, X)\le d_n(M, \, X)\le \lambda_n(M, \, X),
\quad b_n(M, \, X)\le d^n(M, \, X)\le \lambda_n(M, \, X).
\end{align}
In \cite{pietsch1} a notion of strict $s$-numbers of a linear
continuous operator was introduced. Examples of strict $s$-numbers
are the Kolmogorov numbers, the Gelfand numbers, the approximation
numbers and the Bernstein numbers. The Kolmogorov numbers of an
operator $A\in L(X, \, Y)$ are defined by
$$
d_n(A:X\rightarrow Y)=d_n(A(B_X), \, Y),
$$
the approximation numbers, by
$$
a_n(A:X\rightarrow Y)=\inf \{\| A-A_n\| _{X\rightarrow Y}:{\rm
rk}\, A_n\le n\},
$$
the Gelfand numbers, by
$$
c_n(A:X\rightarrow Y)=\inff _{\{x_j^*\}_{j=1}^n\subset X^*} \sup
\{\|Ax\|:\;x\in \cap _{j=1}^n \ker x_j^*, \; \|x\|\le 1\},
$$
the Bernstein numbers, by
$$
b_n(A:X\rightarrow Y)=\sup _{\dim \, L\ge n+1} \sup \{R>0:\;
B_R(0) \subset A(B_X\cap L)\}
$$
(the supremum is taken over the family of subspaces in $X$ of
dimension $n+1$ or larger). Then $b_n(A:X\rightarrow
Y)=b_n(A(B_X), \, Y)$. Heinrich \cite{heinr} proved that if the
operator $A$ is a compact embedding, then
\begin{align}
\label{heinr_result} a_n(A:X\rightarrow Y)=\lambda_n(A(B_X), \,
Y), \quad c_n(A:X\rightarrow Y)=d^n(A(B_X), \, Y).
\end{align}

Let $r\in \N$. For measurable nonnegative functions $u$, $w$ and
$m\in \N$ we define the Riemann -- Liouville operators
\begin{align}
\label{tirmu1} \begin{array}{c} \tilde I_{m,u,w,b}\varphi(t)
=\frac{w(t)}{(m-1)!} \int \limits _t^b (s-t)^{m-1}
u(s)\varphi(s)\, ds,
\\
I_{m,u,w,a} \psi(t) = \frac{w(t)}{(m-1)!}\int \limits _a^t
(t-s)^{m-1} u(s)\psi(s)\, ds.
\end{array}
\end{align}
If $m=1$, they are called two-weighted Hardy-type operators. Given
$r\ge 2$, $1\le m\le r-1$, we set
\begin{align}
\label{tir1} \tilde I^{a,b,m}_{r,u,w} =I_{m,1,w,a} \circ \tilde
I_{r-m,u,1,b}, \quad \tilde I^m_{r,u,w}=\tilde I^{0,1,m}_{r,u,w}.
\end{align}
The criterion of continuity of the operator $\tilde
I^{a,b,m}_{r,u,w}:L_p[a, \, b] \rightarrow L_q[a, \, b]$ was
obtained by Heinig and Kufner \cite{kufner_heinig}. The criterion
of continuity of two-weighted Riemann -- Liouville operators
$I_{m,u,w,a}:L_p[a, \, b] \rightarrow L_q[a, \, b]$ was obtained
by Stepanov \cite{stepanov1, stepanov2}.

Given $1<p<\infty$, we set $p'=\frac{p}{p-1}$. Then, if the
operator $\tilde I^{a,b,m}_{r,u,w}:L_p[a, \, b] \rightarrow L_q[a,
\, b]$ is continuous, then its dual has the form $\tilde
I^{a,b,r-m}_{r,w,u}:L_{q'}[a, \, b] \rightarrow L_{p'}[a, \, b]$.

Let $g$, $v:[a, \, b] \rightarrow (0, \, \infty)$ be measurable
functions, $1<p<\infty$, $1<q<\infty$, $r\ge 2$, $1\le k\le r-1$,
and let the operator $\tilde I^{a,b,k}_{r,g,v}: L_p[a, \, b]
\rightarrow L_q[a, \, b]$ be continuous. Then Heinig and Kufner's
criterion \cite[Theorem 1 and Section 2]{kufner_heinig} implies
that for any $\varepsilon \in (0, \, 1)$
\begin{align}
\label{gv} \begin{array}{c} g\in L_{p'}(a+\varepsilon, \, b], \;\;
v\in L_q(a+\varepsilon, \, b], \\ (s\mapsto (s-a)^kv(s))\in L_q[a,
\, b-\varepsilon), \;\; (s\mapsto (s-a)^{r-k}g(s))\in L_{p'}[a, \,
b-\varepsilon). \end{array}
\end{align}

We set
\begin{align}
\label{w_rk_pg_def} W^{r,k}_{p,g}[a, \, b]=\{\tilde
I^{a,b,k}_{r,g,1}\varphi:\; \|\varphi\|_{L_p[a, \, b]}\le 1\}.
\end{align}
From (\ref{gv}) it follows that the iterated integral
$$
\frac{v(t)}{(k-1)!(r-k-1)!}\int \limits_a^t (t-s)^{k-1}\int
\limits _s^b (\tau -s)^{r-k-1} g(\tau)\varphi(\tau)\, d\tau\, ds
$$
is well-defined for all $\varphi\in L_p[a, \, b]$. We call the set
$W^{r,k}_{p,g}[a, \, b]$ a weighted Sobolev class with
restrictions $f(a)=\dots =f^{(k-1)}(a)=f^{(k)}(b)=\dots
=f^{(r-1)}(b)=0$.

Let
$$
\|f\|_{L_{q,v}[a, \, b]}=\| vf\| _{L_q[a, \, b]}, \quad L_{q,v}[a,
\, b] =\{f: \; \|f\|_{L_{q,v}[a, \, b]}<\infty\}.
$$
We call $L_{q,v}[a, \, b]$ a weighted Lebesgue space.

We set ${\cal W}^{r,k}_{p,g}[a, \, b]={\rm span}\,
W^{r,k}_{p,g}[a, \, b]$ and endow this space with the norm
$\|f\|_{{\cal W}^{r,k}_{p,g}[a, \, b]}=\left\|
\frac{f^{(r)}}{g}\right\|_{L_p[a, \, b]}$.

Throughout we assume that $1<q\le p<\infty$, $g(t)>0$, $v(t)>0$
a.e. and the operator $\tilde I^{a,b,k}_{r,g,v}: L_p[a, \, b]
\rightarrow L_q[a, \, b]$ is compact.

We denote $$h_{(\sigma)} =|h|^{\sigma-1} {\rm sgn}\, h$$ and
consider the following boundary value problem:
\begin{align}
\label{eq_main} \left\{
\begin{array}{l}
x^{(r)}=(-1)^{r-k}g^{p'}y_{(p')}, \\
y^{(r)}=(-1)^k\theta^q v^q x_{(q)}, \\
x(a)=\dots =x^{(k-1)}(a)=x^{(k)}(b)=\dots= x^{(r-1)}(b)=0, \\
y(a)=\dots = y^{(r-k-1)}(a)= y^{(r-k)}(b)= \dots = y^{(r-1)}(b)=0,
\\ \left \| \frac{x^{(r)}}{g} \right\| _{L_p[a, \, b]}=1;
\end{array}
\right.
\end{align}
here $\theta >0$, the functions $x,  \, \dots, \, x^{(k-1)}$, $y,
\, \dots, \, y^{(r-k-1)}$ are locally absolutely continuous on
$[a, \, b)$, and the functions $x^{(k)}, \, \dots, \, x^{(r-1)}$,
$y^{(r-k)}, \, \dots, \, y^{(r-1)}$ are locally absolutely
continuous on $(a, \, b]$.

Let $f$ be a measurable function, $n\in \Z_+$. We say that $f$ has
exactly $n$ points of sign change if there are points
$a=t_0<t_1<\dots <t_n<t_{n+1}=b$ such that
\begin{enumerate}
\item for any $1\le j\le n$, $t$, $s\in [t_{j-1}, \, t_j]$
the inequality $f(t)f(s)\ge 0$; in addition, the set $\{t\in
[t_{j-1}, \, t_j]:\; f(t)\ne 0\}$ has a positive measure for each
$j\in \{1, \, \dots, \, n\}$;
\item for any $1\le j\le n-1$, $t\in [t_{j-1}, \, t_j]$, $s\in [t_j, \,
t_{j+1}]$ the inequality $f(t)f(s)\le 0$ holds.
\end{enumerate}
We say that $f$ has no more than $n$ points of sign change if $f$
has exactly $m$ points of sign change with $0\le m\le n$.

Let $x\in {\cal W}^r_{p,g}[a, \, b]$, $y\in {\cal W}^r_{q',v}[a,
\, b]$, $\theta>0$, $n\in \Z_+$. We say that $(x, \, y, \, \theta)
\in SP_n$ (correspondingly, $(x, \, y, \, \theta) \in
\widetilde{SP}_n$), if (\ref{eq_main}) holds and the function $x$
has exactly $n$ (correspondingly, no more that $n$) points of sign
change; we call $x$ {\it the spectral function}. Denote by $sp_n$
(correspondingly, by $\widetilde{sp}_n$) the set of numbers
$\theta>0$ such that $(x, \, y, \, \theta) \in SP_n$
(correspondingly, $(x, \, y, \, \theta) \in \widetilde{SP}_n$) for
some $x\in {\cal W}^r_{p,g}[a, \, b]$, $y\in {\cal W}^r_{q',v}[a,
\, b]$. We set $\overline{\theta}_n=\sup sp_n$, $\tilde\theta
_n=\sup \widetilde{sp}_n$.

\begin{Trm}
\label{main_trm} Let $1<q\le p<\infty$, $g(t)>0$, $v(t)>0$ a.e.,
the operator $\tilde I^{a,b,k}_{r,g,v}:L_p[a, \, b] \rightarrow
L_q[a, \, b]$ is compact. Then
\begin{align}
\label{main_eq} d_n(W^{r,k}_{p,g}[a, \, b], \, L_{q,v}[a, \, b]) =
\lambda_n(W^{r,k}_{p,g}[a, \, b], \, L_{q,v}[a, \,
b])=d^n(W^{r,k}_{p,g}[a, \, b], \, L_{q,v}[a, \, b])
=\overline{\theta}_n^{\, -1}.
\end{align}
Moreover, the widths strictly decrease with respect to $n$. If
$p=q$, then the set $sp_n$ is a singleton and
\begin{align}
\label{bern_dn} b_n(W^{r,k}_{p,g}[a, \, b], \, L_{p,v}[a, \,
b])=d_n(W^{r,k}_{p,g}[a, \, b], \, L_{p,v}[a, \, b])
=\overline{\theta}_n^{\, -1}.
\end{align}
\end{Trm}
The equality (\ref{bern_dn}) is proved similarly as in \cite[p.
392--393]{vas_spn} and \cite[p. 24--25]{pinkus_constr}. Notice
that in the paper of Edmunds and Lang \cite{edm_lang_in_an} the
coincidence of all strict $s$-numbers of a compact two-weighted
Hardy operator was proved for $p=q$.

The paper is organized as follows. In Section 2 we apply the
method from the paper of Buslaev and Tikhomirov \cite{busl_tikh}
and prove (\ref{main_eq}) with $\tilde \theta_n$ instead of
$\overline{\theta}_n$. In proof of the upper estimate we use the
integration operator with the kernel $G(t, \, \tau, \, \xi, \,
\eta)$ defined by formula (\ref{gtt_xi}), which differ from the
kernel from \cite{busl_tikh}.

In Sections 3--5 we prove that the widths are strictly decreasing.
The method of the proof is the same as in \cite{vas_spn}. From the
strict monotonicity of widths we now obtain (\ref{main_eq}) with
$\overline{\theta}_n$. Everywhere in \S 2--5 we write only the
fragments of proof that differ from arguments in \cite{busl_tikh}
and \cite{vas_spn}. Notice that everywhere in \S 2--5 without loss
of generality we may assume that $[a, \, b]=[0, \, 1]$.

In Section 6 we apply the main result in estimating the spectral
numbers for (\ref{eq_main}) with weights $g(x)=x^{-\beta_g}|\ln
x|^{-\alpha_g}\rho_g(|\ln x|)$, $v(x) = x^{-\beta_v} |\ln
x|^{-\alpha_v} \rho_v(|\ln x|)$; here $\rho_g$, $\rho_v$ are
``slowly varying'' functions, $\alpha_g+\alpha_v = r+\frac
1q-\frac 1p$. The order estimates immediately follow from
(\ref{main_eq}) and results of the paper \cite{vas_alg_an}. In
addition, if the parameters are such that $gv \in L_\varkappa$
with $\frac{1}{\varkappa} =r +\frac 1q-\frac 1p$, then asymptotics
of widths and spectral numbers is obtained.

For piecewise continuous weights asymptotics of widths for $p\ge
q$ was obtained by Buslaev \cite{busl_aa}, and for $g\in L_{p'}[a,
\, b]$, $v\in L_q[a, \, b]$, $r=1$, by Edmunds and Lang
\cite{edm_lang} (see also Buslaev's paper \cite{busl_dan} for
$r\in \N$ and some supplementary conditions on weights). For
$p=q$, $r=1$ and more general conditions on weights asymptotics
was obtained by Evans, Harris, Lang, Edmunds and Kerman in
\cite{ev_har_lang, lang_j_at1, lang_remainder}, and for $p\ge q$,
$r\in \N$, in \cite{vas_m_sb}.

The problem of determining exact values of widths was studied in
papers of Kolmogorov, Tikhomirov, Babadzanov, Makovoz, Ligun,
Pinkus and others \cite{kolmog_dn, tikh_babaj, bibl6, bib_makovoz,
tikh_nvtp_art, ligun_aa, pinkus_79} (for details, see
\cite{busl_tikh}). In paper of Malykhin \cite{malykhin_y_v} the
asymptotics of $$d_{n+r}(W^r_\infty[-1, \, 1], \, C[-1, \,
1])/d_{r}(W^r_\infty[-1, \, 1], \, C[-1, \, 1]), \quad r\to
\infty,$$ was obtained. In papers of Babenko and Parfinovich
\cite{bab_parf, bab_parf1} the subspaces of splines of different
defect were studied as extremal subspaces in the problem of
calculating widths.

\section{The application of Buslaev and Tikhomirov's method}

We may assume that $[a, \, b]=[0, \, 1]$.

First we prove that
\begin{align}
\label{dn_low_est} d_n(W^{r,k}_{p,g}[0, \, 1], \, L_{q,v}[0, \,
1]) \ge \tilde{\theta}_n^{\, -1}, \quad d^n(W^{r,k}_{p,g}[0, \,
1], \, L_{q,v}[0, \, 1]) \ge \tilde{\theta}_n^{\, -1}.
\end{align}
This together with (\ref{bndnln}) yield that the similar lower
estimate holds for linear widths.

The system (\ref{eq_main}) with $[a, \, b]=[0, \, 1]$ can be
written as the system of integral equations:
\begin{align}
\label{eq_main_int} \left\{
\begin{array}{l}
x=\tilde I^k_{r,g,1}(g^{p'-1}y_{(p')}), \\
y=\theta^q \tilde I^{r-k}_{r,v,1}(v^{q-1} x_{(q)}),
\\ \left \| g^{p'-1} y_{(p')} \right\| _{L_p[0, \, 1]}=1.
\end{array}
\right.
\end{align}
Notice that $\|g^{p'-1}y_{(p')}\|_{L_p[0, \, 1]} =
\|gy\|_{L_{p'}[0, \, 1]}$.

Let $u$ be a piecewise continuous function with values $\pm 1$.
The Buslaev's iteration process \cite[\S 6]{busl_tikh} is written
as follows:
$$
x_0(\cdot, \, u)=\tilde I^k_{r,g,1}u(\cdot);
$$
for $m\in \N$ we set
\begin{align}
\label{ym_u} y_m(\cdot , \, u) = \theta_{m-1}^q(u) \tilde
I^{r-k}_{r,v,1} [(vx_{m-1})_{(q)}],
\end{align}
\begin{align}
\label{xm_u} x_m(\cdot, \, u) =\tilde I^k_{r,g,1}[(gy_m)_{(p')}];
\end{align}
here $\theta_{m-1}(u)>0$ is chosen so that
\begin{align}
\label{gymlp1} \|gy_m\| _{L_{p'}[0, \, 1]}=1
\end{align}
(it is possible, since $v(\cdot)x_{m-1}(\cdot, \, u)\ne 0$ a.e.,
which can be easily proved by induction). Since the operator
$\tilde I^k_{r,g,v}:L_p[0, \, 1] \rightarrow L_q[0, \, 1]$ and its
dual $\tilde I^{r-k}_{r,v,g}:L_{q'}[0, \, 1] \rightarrow L_{p'}[0,
\, 1]$ are compact, by induction method it can be proved that
$x_m(\cdot, \, u)\in L_{q,v}[0, \, 1]$. Indeed, $u\in L_p[0, \,
1]$. Therefore, $x_0(\cdot, \, u)\in L_{q,v}[0, \, 1]$. Let
$x_{m-1}(\cdot, \, u) \in L_{q,v}[0, \, 1]$. Then
\begin{align}
\label{xm1} \begin{array}{c} (v(\cdot)x_{m-1}(\cdot, \, u))_{(q)}
\in L_{q'}[0, \, 1] \; \Rightarrow \; g(\cdot)y_m(\cdot, \, u) \in
L_{p'}[0, \, 1]\; \Rightarrow \\ \Rightarrow\; (g(\cdot)y_m(\cdot,
\, u))_{(p')} \in L_p[0, \, 1] \; \Rightarrow \;x_m(\cdot, \,
u)\in L_{q,v}[0, \, 1].
\end{array}
\end{align}

The following lemma is similar to Lemma 1 from \cite[\S
6]{busl_tikh}.
\begin{Lem}
\label{l1} The following estimates hold:
$$
\|x_{m-1}(\cdot, \, u)\|_{L_{q,v}[0, \, 1]} \le [\theta
_{m-1}(u)]^{-1} \le \|x_m(\cdot, \, u)\|_{L_{q,v}[0, \, 1]}.
$$
\end{Lem}
\begin{proof}
The arguments are almost similar as in \cite{busl_tikh}, but
instead of integration by parts we apply the Fubini theorem.
Everywhere we denote $x_j=x_j(\cdot, u)$, $\theta_j=\theta_j(u)$.
We have
$$
1=\int \limits _0^1 \left| \frac{x_m^{(r)}}{g} \right|^p\, dt
\stackrel{(\ref{xm_u}), (\ref{gymlp1})}{=} \int \limits _0^1
\frac{x_m^{(r)}}{g} \left(\frac{x_m^{(r)}}{g}\right)_{(p)}\, dt
\stackrel{(\ref{xm_u})}{=} \int \limits _0^1 (-1)^{r-k}
\frac{x_m^{(r)}}{g} gy_m\, dt=:I.
$$

Let
\begin{align}
\label{phi_psi} \varphi =\frac{x_m^{(r)}}{g}
\stackrel{(\ref{xm_u})}{=} (gy_m)_{(p')}, \quad \psi =
v^{q-1}[x_{m-1}]_{(q)}.
\end{align}
Then by (\ref{xm1}) we get
$$
\varphi \in L_p[0, \, 1], \;\; \psi \in L_{q'}[0, \, 1], \;\; gy_m
\in L_{p'}[0, \, 1],
$$
$$
I \stackrel{(\ref{ym_u})}{=} [\theta_{m-1}]^q\int \limits _0^1
\varphi \cdot \tilde I^{r-k}_{r,v,g}\psi
\,dt\stackrel{(\ref{tirmu1}), (\ref{tir1})}{=}
$$
$$
=\frac{[\theta_{m-1}]^q}{(k-1)!(r-k-1)!}\int \limits _0^1
\varphi(t) g(t) \int _0^t (t-s)^{r-k-1}\int \limits
_s^1(\tau-s)^{k-1} v(\tau) \psi(\tau)\, d\tau\, ds\, dt=
$$
$$
=\frac{[\theta_{m-1}]^q}{(k-1)!(r-k-1)!}\int \limits _0^1
\psi(\tau) v(\tau) \int \limits_0^\tau (\tau-s)^{k-1} \int \limits
_s^1(t-s)^{r-k-1}g(t)\varphi(t)\, dt\, ds\, d\tau
\stackrel{(\ref{tirmu1}), (\ref{tir1})}{=}
$$
$$
=[\theta_{m-1}]^q\int \limits _0^1 \psi\cdot \tilde
I^k_{r,g,v}\varphi \, d\tau \stackrel{(\ref{xm_u}),
(\ref{phi_psi})}{=} [\theta_{m-1}]^q\int \limits_0^1
v^{q-1}[x_{m-1}]_{(q)} \cdot vx_m\, dt\le
$$
$$
\le [\theta_{m-1}]^q \|x_{m-1}\|_{L_{q,v}[0, \, 1]}^{q-1}
\|x_m\|_{L_{q,v}[0, \, 1]}.
$$
Thus,
\begin{align}
\label{1le} 1\le [\theta_{m-1}]^q \|x_{m-1}\|_{L_{q,v}[0, \,
1]}^{q-1} \|x_m\|_{L_{q,v}[0, \, 1]}.
\end{align}

Further,
$$
1 \stackrel{(\ref{xm_u}), (\ref{gymlp1})}{=}\left\|
\frac{x^{(r)}_{m-1}}{g}\right\| _{L_p[0, \, 1]} \left\|
\frac{x^{(r)}_m}{g} \right\|^{p-1}_{L_p[0, \, 1]} \ge \int \limits
_0^1 \left( \frac{x_m^{(r)}}{g}\right)_{(p)}
\frac{x_{m-1}^{(r)}}{g} \, dt \stackrel{(\ref{xm_u})}{=} \int
\limits_0^1 gy_m \frac{x_{m-1}^{(r)}}{g}\, dt=:J.
$$
Applying (\ref{ym_u}) together with the Fubini theorem, we get
that
$$
J=[\theta_{m-1}]^q \int \limits_0^1 v^{q-1}[x_{m-1}]_{(q)} \cdot
\tilde I^k_{r,g,v} \frac{x_{m-1}^{(r)}}{g}\, dt=[\theta_{m-1}]^q
\|x_{m-1}\|^q_{L_{q,v}[0, \, 1]}.
$$
Hence, $\|x_{m-1}\|_{L_{q,v}[0, \, 1]} \le [\theta_{m-1}]^{-1}$.
This together with (\ref{1le}) yields that $[\theta_{m-1}]^{-1}\le
\|x_m\|_{L_{q,v}[0, \, 1]}$.
\end{proof}
Repeating arguments from \cite{busl_tikh} (see Lemma 2 from
Section 6 and its corollary, as well as Lemma 1 from Section 8),
we get the following assertions.
\begin{enumerate}
\item One can choose a subsequence of $\{x_m(\cdot, \, u)\}_{m\in \N}$
convergent to a spectral function in the space $L_{q,v}[0, \, 1]$.
\item The set $SP_0$ is non-empty.
\item The estimates (\ref{dn_low_est}) hold.
\end{enumerate}

Now we prove the inequality
\begin{align}
\label{dn_up_est} \lambda_n(W^{r,k}_{p,g}[0, \, 1], \, L_{q,v} [0,
\, 1]) \le \tilde \theta_n^{-1}.
\end{align}
From (\ref{bndnln}) it follows that the same inequality holds for
Kolmogorov and Gelfand widths.

Similarly as in \cite[p. 367]{vas_spn} we prove that the set
$\widetilde{sp}_n$ is closed.

Let $(\overline{x}, \, \overline{y}, \, \overline{\theta}) \in
SP_m$, $0\le m\le n$. We show that
\begin{align}
\label{up_est} \lambda_n(W^{r,k}_{p,g}[0, \, 1], \, L_{q,v}[0, \,
1])\le \overline{\theta}^{\, -1}.
\end{align}
Since the operator $\tilde I^k_{r,g,v}: L_p[0, \, 1] \rightarrow
L_q[0, \, 1]$ is compact, by (\ref{w_rk_pg_def}) and Heinrich's
results (\ref{heinr_result}) it is sufficient to prove that
$$
a_n(\tilde I^k_{r,g,v}: L_p[0, \, 1] \rightarrow L_q[0, \,
1])\le \overline{\theta}^{\, -1}
$$
(recall that $a_n$ are the approximation numbers).

Applying the Rolle theorem and repeating arguments from
Proposition 1 in \cite{vas_spn} (see also \cite[p.
1594]{busl_tikh}), we obtain the following assertion.
\begin{Sta}
\label{per_zn} Let $(\overline{x}, \, \overline{y}, \,
\overline{\theta}) \in SP_m$, $m\in \Z_+$. Then the following
assertions hold.
\begin{enumerate}
\item The functions $\overline{x}$, $\overline{y}$ and their derivatives $\overline{x}^{(j)}$,
$\overline{y}^{(j)}$, $1\le j\le r-1$, have exactly $m$ zeros on
$(0, \, 1)$.
\item Let
$$
0<\xi_1<\dots <\xi_m<1, \quad 0<\eta_1<\dots <\eta_m<1, \quad
\overline{x}(\xi_i)=0, \quad \overline{y}(\eta_i)=0, \quad 1\le
i\le m.
$$
Then $\xi_i$ are points of sign change of the function
$\overline{x}$, $\eta_i$ are points of sign change of the function
$\overline{y}$.
\item We have $\dot{\overline{x}}(\xi_i)\ne 0$, $\dot{\overline{y}}(\eta_i)\ne
0$, $1\le i\le m$.
\end{enumerate}
\end{Sta}
Now we obtain the alternation condition for points $\xi_i$,
$\eta_i$.
\begin{Sta}
\label{peremez1} Let the points $\xi_i$, $\eta_i$ $(1\le i\le m)$
be such as in Proposition \ref{per_zn}. Then
\begin{align}
\label{xi_j_eta} \eta_{j+k-r}<\xi_j<\eta_{j+k}.
\end{align}
\end{Sta}
\begin{proof}
From the Rolle theorem and the condition $\overline{x}(0)=\dots
=\overline{x}^{(k-1)}(0)=0$ it follows that the function
$\overline{x}^{(k)}$ has $m$ zeroes $0<\mu_1<\dots <\mu_m<1$ on
the interval $(0, \, 1)$; moreover, the alternation condition
$\mu_j<\xi_j<\mu_{j+k}$ holds. By (\ref{eq_main}), the points of
sign change of the function $\overline{x}^{(r)}$ coincide with
zeroes of the function $\overline{y}$ on $(0, \, 1)$. By the Rolle
theorem and the condition $\overline{x}^{(k)}(1)=\dots
=\overline{x}^{(r-1)}(1)=0$, the alternation condition
$\mu_j<\eta_j<\mu_{j+r-k}$ holds. Hence,
$\xi_j<\mu_{j+k}<\eta_{j+k}$, $\xi_j>\mu_j>\eta_{j+k-r}$. This
completes the proof of (\ref{xi_j_eta}).
\end{proof}

{\bf Definition of the kernel $G(t, \, \tau, \, \xi, \, \eta)$.}
Let $\mu=\{\mu_j\}_{j=1}^J\subset [0, \, 1]$,
$\nu=\{\nu_j\}_{j=1}^J\subset [0, \, 1]$, $l\in \N$. We set
\begin{align}
\label{kl_munu} K_l(\mu, \, \nu) = \begin{vmatrix}
(\mu_1-\nu_1)_+^{l-1} & (\mu_2-\nu_1)_+^{l-1} & \dots &
(\mu_J-\nu_1)_+^{l-1}
\\ (\mu_1-\nu_2)_+^{l-1} & (\mu_2-\nu_2)_+^{l-1} & \dots &
(\mu_J-\nu_2)_+ ^{l-1} \\ \dots & \dots & \dots & \dots
\\ (\mu_1-\nu_J)_+^{l-1} & (\mu_2-\nu_J)_+^{l-1} &
\dots & (\mu_J-\nu_J)_+^{l-1}\end{vmatrix}
\end{align}
(if $l=1$, then we set $x_+^0=0$ for $x<0$ and $x_+^0=1$ for $x\ge
0$). In \cite[Lemma 9.2]{karlin_studden} it was proved that if
$\mu_1<\dots <\mu_J$, $\nu_1< \dots <\nu_J$, then $K_l(\mu, \,
\nu)\ge 0$; in addition, if the alternation condition
$\nu_j<\mu_j<\nu_{j+l}$ holds, then $K_l(\mu, \, \nu)>0$.

Let now $r\ge 2$, $1\le k\le r-1$, $\mu=\{\mu_j\}_{j=1}^J \subset
[0, \, 1]$, $\nu=\{\nu_j\}_{j=1}^J\subset [0, \, 1]$. We set
\begin{align}
\label{krk} K_{r,k}(\mu, \, \nu) = \int \limits _{[0, \, 1]^J}
K_k(\mu, \, \alpha) K_{r-k}(\nu, \, \alpha)\, d\alpha,
\end{align}
where $\alpha=(\alpha_1, \, \dots , \, \alpha_J)$.
\begin{Lem}
\label{lem_kg0} Let $0<\mu_1<\dots <\mu_J<1$, $0<\nu_1< \dots
<\nu_J<1$. Then $K_{r,k}(\mu, \, \nu)\ge 0$. In addition, if the
alternation condition
\begin{align}
\label{nu_jkr} \nu_{j+k-r}<\mu _j < \nu_{j+k},
\end{align}
holds, then $K_{r,k}(\mu, \, \nu)>0$.
\end{Lem}
\begin{proof}
There exists a rearrangement $\sigma\in S_J$ such that
$\alpha_{\sigma(1)}\le \dots \le \alpha _{\sigma(J)}$. We set
$\tilde \alpha =(\alpha_{\sigma(1)}, \, \dots, \, \alpha
_{\sigma(J)})$. Then $K_k(\mu, \, \tilde\alpha)\ge 0$,
$K_{r-k}(\nu, \, \tilde\alpha)\ge 0$. Rearranging lines in the
determinant, we obtain that
\begin{align}
\label{k_k_mu} K_k(\mu, \, \alpha) K_{r-k}(\nu, \,
\alpha)=(-1)^\sigma K_k(\mu, \, \tilde\alpha)\cdot (-1)^\sigma
K_{r-k}(\nu, \, \tilde\alpha)= K_k(\mu, \, \tilde\alpha)\cdot
K_{r-k}(\nu, \, \tilde\alpha) \ge 0.
\end{align}
It remains to integrate over $\alpha$.

Now let the alternation condition hold. We show that there exists
a sequence $0<\alpha_1<\dots <\alpha_J<1$ such that
$\alpha_j<\mu_j<\alpha_{j+k}$, $\alpha_j<\nu_j<\alpha_{j+r-k}$.
This together with (\ref{k_k_mu}) implies that $K_{r,k}(\mu, \,
\nu)>0$.

We take a sufficiently small $\varepsilon>0$ and set $\alpha_j
=\min\{\mu_j, \, \nu_j\} -\varepsilon$. Since $\mu_j<\mu_{j+1}$,
$\nu_j<\nu_{j+1}$, we have $\alpha_j<\alpha_{j+1}$ for small
$\varepsilon>0$. The inequalities $\alpha_j<\mu_j$,
$\alpha_j<\nu_j$ hold by construction. We show that
$\mu_j<\alpha_{j+k}$ and $\nu_j<\alpha_{j+r-k}$. To this end, it
is sufficient to check the inequalities $\mu_j<\min\{\mu_{j+k}, \,
\nu_{j+k}\}$, $\nu_j<\min\{\mu_{j+r-k}, \, \nu_{j+r-k}\}$. They
hold by the strict monotonicity of $\{\mu_i\}_{i=1}^J$,
$\{\nu_i\}_{i=1}^J$ and the alternation condition (\ref{nu_jkr}).
\end{proof}
\begin{Lem}
\label{det} Let $\varphi, \, \psi\in L_\infty([0, \, 1]^2)$. For
$\mu=\{\mu_j\}_{j=1}^J$, $\nu=\{\nu_j\}_{j=1}^J$,
$\alpha=\{\alpha_j\}_{j=1}^J \subset [0, \, 1]$ we set
$$
\Phi(\mu, \, \alpha) =\begin{vmatrix} \varphi(\mu_1, \, \alpha_1)
& \varphi(\mu_2, \, \alpha_1) & \dots & \varphi(\mu_J, \,
\alpha_1)
\\ \varphi(\mu_1, \, \alpha_2)
& \varphi(\mu_2, \, \alpha_2) & \dots & \varphi(\mu_J, \,
\alpha_2) \\ \dots & \dots & \dots & \dots \\ \varphi(\mu_1, \,
\alpha_J) & \varphi(\mu_2, \, \alpha_J) & \dots & \varphi(\mu_J,
\, \alpha_J)
\end{vmatrix},
$$
$$
\Psi(\nu, \, \alpha) =\begin{vmatrix} \psi(\nu_1, \, \alpha_1) &
\psi(\nu_2, \, \alpha_1) & \dots & \psi(\nu_J, \, \alpha_1)
\\ \psi(\nu_1, \, \alpha_2)
& \psi(\nu_2, \, \alpha_2) & \dots & \psi(\nu_J, \, \alpha_2) \\
\dots & \dots & \dots & \dots \\ \psi(\nu_1, \, \alpha_J) &
\psi(\nu_2, \, \alpha_J) & \dots & \psi(\nu_J, \, \alpha_J)
\end{vmatrix},
$$
$$
\Lambda(\mu, \, \nu) =\int \limits _{[0, \, 1]^J}\Phi(\mu, \,
\alpha)\Psi(\nu, \, \alpha)\, d\alpha, \quad H(t, \, \tau) =\int
\limits _0^1 \varphi(t, \, s)\psi(\tau, \, s)\, ds.
$$
Then
$$
\Lambda(\mu, \, \nu)=J!\begin{vmatrix} H(\mu_1, \, \nu_1) &
H(\mu_2, \, \nu_1) & \dots & H(\mu_J, \, \nu_1)
\\ H(\mu_1, \, \nu_2)
& H(\mu_2, \, \nu_2) & \dots & H(\mu_J, \, \nu_2) \\
\dots & \dots & \dots & \dots \\ H(\mu_1, \, \nu_J) & H(\mu_2, \,
\nu_J) & \dots & H(\mu_J, \, \nu_J)
\end{vmatrix}.
$$
\end{Lem}
The assertion is a consequence of the formula (3.12) in
\cite{karlin_studden}. For convenience, we give the proof.
\begin{proof}
We have
$$
\Lambda(\mu, \, \nu)=\sum \limits _{\sigma\in S_J} \sum \limits
_{\pi\in S_J} (-1)^{\sigma+\pi} \int \limits _{[0, \, 1]^J} \prod
_{i=1}^J\varphi(\mu_{\sigma(i)}, \, \alpha_i ) \prod _{j=1}^J
\psi(\nu_{\pi(j)}, \, \alpha _j)\, d\alpha_1\dots d\alpha_J=
$$
$$
=\sum \limits _{\sigma\in S_J} \sum \limits _{\pi\in S_J}
(-1)^{\sigma+\pi} \prod _{j=1}^J \left(\int \limits_0^1
\varphi(\mu_{\sigma(j)}, \, \alpha_j)\psi(\nu_{\pi(j)}, \, \alpha
_j)\, d\alpha_j\right)=
$$
$$
=\sum \limits _{\sigma\in S_J} \sum \limits _{\pi\in S_J}
(-1)^{\sigma+\pi} \prod _{j=1}^J H(\mu_{\sigma(j)}, \,
\nu_{\pi(j)})=\sum \limits _{\sigma\in S_J} \sum \limits _{\pi\in
S_J}(-1)^{\sigma+\pi}\prod _{i=1}^J H(\mu_i, \,
\nu_{\pi(\sigma^{-1}(i))})=
$$
$$
=\sum \limits _{\sigma\in S_J} \sum \limits _{\rho\in
S_J}(-1)^{\rho}\prod _{i=1}^J H(\mu_i, \, \nu_{\rho(i)}) =J! \det
(H(\mu_i, \, \nu_j))_{1\le i,j\le J}.
$$
This completes the proof.
\end{proof}

Now let $m\in \Z_+$, $0<\xi_1<\dots <\xi_m<1$, $0<\eta_1<\dots
<\eta_m<1$, and let the alternation condition
$\eta_{j+k-r}<\xi_j<\eta_{j+k}$ hold. For $t\in [0, \, 1]$, $\tau
\in [0, \, 1]$ we set
\begin{align}
\label{gtt_xi} G(t, \, \tau, \, \xi, \, \eta)
=\frac{1}{(k-1)!(r-k-1)!m} \frac{K_{r,k}(t, \, \xi_1, \, \dots, \,
\xi_m, \, \tau, \, \eta_1, \, \dots, \, \eta_m)} {K_{r,k}(\xi_1,
\, \dots, \, \xi_m, \, \eta_1, \, \dots, \, \eta_m)}.
\end{align}
By Lemma \ref{lem_kg0} and the alternation condition, the
denominator is strict positive. We denote
\begin{align}
\label{htt} H(t, \, \tau) =\int \limits_0^1
(t-s)_+^{k-1}(\tau-s)_+^{r-k-1}\, ds.
\end{align}
Then for a fixed $t\in (0, \, 1)$
\begin{align}
\label{h_est} H(t, \, \tau) \underset{r,k,t}{\asymp} 1
\quad\text{if } \tau\text{ is in a left neighborhood of }1,
\end{align}
\begin{align}
\label{h_est1} H(t, \, \tau) \underset{r,k,t}{\asymp} \tau^{r-k}
\quad \text{if }\tau\text{ is in a right neighborhood of }0.
\end{align}

\begin{Lem}
\label{g_prop} The following assertions hold.
\begin{enumerate}
\item For any $t\in [0, \, 1]$, $\tau\in [0, \, 1]$, $1\le j\le m$ the
equalities $G(\xi_j, \, \tau, \, \xi, \, \eta)=0$ and $G(t, \,
\eta_j, \, \xi, \, \eta)=0$ hold; in particular,
\begin{align}
\label{il01g} \int \limits _0^1 G(\xi_j, \, \tau, \, \xi, \,
\eta)\varphi(\tau)\, d\tau =0, \quad \varphi\in L_1^{{\rm loc}}(0,
\, 1),
\end{align}
\begin{align}
\label{il01g_eta} \int \limits _0^1 G(t, \, \eta_j, \, \xi, \,
\eta)\varphi(\tau)\, d\tau =0, \quad \varphi\in L_1^{{\rm loc}}(0,
\, 1).
\end{align}
\item If $x\in W^{r,k}_{p,g}[0, \, 1]$, then
\begin{align}
\label{xt1rk} x(t) - (-1)^{r-k}\int \limits_0^1 G(t, \, \tau, \,
\xi, \, \eta) x^{(r)}(\tau)\, d\tau =\sum \limits _{j=1}^m c_j(x)
H(t, \, \eta_j),
\end{align}
where $x\mapsto c_j(x)$ are linear continuous functionals on
${\cal W}^{r,k}_{p,g}[0, \, 1]$.
\item Let $(\overline{x}, \, \overline{y}, \, \overline{\theta}) \in
SP_m$, let $0<\xi_1<\dots<\xi_m<1$ be zeroes of $\overline{x}$,
and let $0<\eta_1< \dots <\eta_m<1$ be zeroes of $\overline{y}$.
Then
\begin{align}
\label{ovr_x} \overline{x}(t)=(-1)^{r-k}\int \limits_0^1 G(t, \,
\tau, \, \xi, \, \eta) \overline{x}^{(r)}(\tau)\, d\tau, \quad
\overline{y}(\tau) = (-1)^k\int \limits_0^1 G(t, \, \tau, \, \xi,
\, \eta) \overline{y}^{(r)}(t) \, dt.
\end{align}
In addition,
\begin{align}
\label{xtgt} (-1)^{r-k}\overline{x}(t)G(t, \, \tau, \, \xi, \,
\eta) \overline{x}^{(r)}(\tau) \ge 0.
\end{align}
\end{enumerate}
\end{Lem}
\begin{proof}
From Lemma \ref{det} and (\ref{kl_munu}), (\ref{krk}),
(\ref{gtt_xi}) it follows that
\begin{align}
\label{gttxi_eta} G(t, \, \tau, \, \xi, \, \eta) =
\frac{1}{C(k-1)!(r-k-1)!}
\begin{vmatrix} H(t, \, \tau) & H(\xi_1, \, \tau) & \dots & H(\xi_m, \,
\tau) \\ H(t, \, \eta_1) & H(\xi_1, \, \eta_1) & \dots & H(\xi_m,
\, \eta_1) \\ \dots & \dots & \dots & \dots \\ H(t, \, \eta_m) &
H(\xi_1, \, \eta_m) & \dots & H(\xi_m, \, \eta_m)
\end{vmatrix},
\end{align}
\begin{align}
\label{chg0} C=\begin{vmatrix}  H(\xi_1, \, \eta_1) & \dots &
H(\xi_m, \, \eta_1) \\ \dots & \dots & \dots \\  H(\xi_1, \,
\eta_m) & \dots & H(\xi_m, \, \eta_m)
\end{vmatrix} >0
\end{align}
by Proposition \ref{peremez1} and Lemma \ref{lem_kg0}. This
implies the first assertion of Lemma.

Let us prove the second assertion. Indeed, from (\ref{htt}),
(\ref{gttxi_eta}) and (\ref{chg0}) we get that
$$
(-1)^{r-k}\int \limits_0^1 G(t, \, \tau, \, \xi, \, \eta)
x^{(r)}(\tau)\, d\tau=\sum \limits _{j=1}^m \tilde c_j(x) H(t, \,
\eta_j)+
$$
$$
+\frac{(-1)^{r-k}}{(k-1)!(r-k-1)!} \int \limits _{[0, \, 1]}
\left(\int \limits _{[0, \, 1]} (t-s)_+^{k-1}(\tau-s)_+^{r-k-1}\,
ds\right) x^{(r)}(\tau)\, d\tau=
$$
$$
=\sum \limits _{j=1}^m \tilde c_j(x) H(t, \, \eta_j)+$$$$+
\frac{(-1)^{r-k}}{(k-1)!(r-k-1)!} \int \limits_0^t (t-s)^{k-1}
\int \limits _s^1 (\tau-s)^{r-k-1} x^{(r)}(\tau)\, d\tau\, ds
\stackrel{(\ref{w_rk_pg_def})}{=}
$$
$$
=\sum \limits _{j=1}^m \tilde c_j(x) H(t, \, \eta_j)+x(t),
$$
where $\tilde c_j(x)$ are linear functionals. We set
$c_j(x)=-\tilde c_j(x)$. The values $c_j(x)$ are some linear
combinations of the integrals
$$
\int \limits _0^1 H(\xi_i, \, \tau)x^{(r)}(\tau)\, d\tau=\int
\limits _0^1 H(\xi_i, \, \tau)g(\tau)\varphi(\tau)\, d\tau, \quad
\varphi =\frac{x^{(r)}}{g}.
$$
From (\ref{gv}), (\ref{h_est}) and (\ref{h_est1}) it follows that
the right part continuously depends on $\varphi\in L_p[0, \, 1]$.
This implies the continuity of the functionals $c_j$ on the space
${\cal W}^{r,k}_{p,g}[0, \, 1]$.

Let us prove the third assertion. By the second assertion, which
is already proved,
\begin{align}
\label{ovr_v_t} \overline{x}(t) - (-1)^{r-k}\int \limits_0^1 G(t,
\, \tau, \, \xi, \, \eta) \overline{x}^{(r)}(\tau)\, d\tau =\sum
\limits _{j=1}^m c_j(\overline{x}) H(t, \, \eta_j).
\end{align}
By the condition of assertion 3 in Lemma, $\overline{x}(\xi_i)=0$.
This together with (\ref{il01g}) implies that the left-hand side
of (\ref{ovr_v_t}) equals to zero in points $\xi_i$. Hence,
$$
\sum \limits _{j=1}^m c_j(\overline{x}) H(\xi_i, \, \eta_j)=0,
\quad 1\le i\le m.
$$
By (\ref{chg0}), we have $c_j(\overline{x})=0$.

The second equality (\ref{ovr_x}) can be proved similarly (here we
apply (\ref{il01g_eta})).

Let us prove (\ref{xtgt}). If $t<\xi_1<\dots<\xi_m$,
$\tau<\eta_1<\dots <\eta_m$, then $G(t, \, \tau, \, \xi, \,
\eta)\ge 0$ by Lemma \ref{lem_kg0} and (\ref{gtt_xi}). Let $t\in
(\xi_i, \, \xi_{i+1})$, $\tau \in (\eta_j, \, \eta_{j+1})$. Then
$$G(t, \, \tau, \, \xi, \, \eta)\cdot (-1)^{i+j} \ge 0.$$ Since
$\xi_i$ are points of sign change of the function $\overline{x}$,
and $\eta_j$ are the points of sign change of the function
$\overline{x}^{(r)}$ by the first equation (\ref{eq_main}), we get
that the sign of $\overline{x}(t)G(t, \, \tau, \, \xi, \, \eta)
\overline{x}^{(r)}(\tau)$ is constant. Therefore, it is sufficient
to prove the inequality $(-1)^{r-k}\overline{x}(t) \overline{x}
^{(r)}(\tau)\ge 0$ for $t$, $\tau$ from a small neighborhood of
zero. Without loss of generality we may assume that $\overline{x}
^{(r)}(\tau)\ge 0$ in a neighborhood of zero. Then $(-1)^m
\overline{x} ^{(r)}(\tau)\ge 0$ in a neighborhood of 1. Therefore,
$(-1)^{m+r-k} \overline{x} ^{(k)}(\tau)\ge 0$ in a neighborhood of
1, which implies that $(-1)^{r-k}\overline{x} ^{(k)}(\tau)\ge 0$
in a neighborhood of 0. Hence, $(-1)^{r-k}\overline{x} (\tau)\ge
0$ in a neighborhood of 0.
\end{proof}
The further arguments in obtaining the upper estimate are the same
as in \cite[\S 7]{busl_tikh}, \cite[p. 389--390]{vas_spn}.

\section{Some properties of the kernel $G(t, \, \tau, \, \xi, \, \eta)$.}
Let $0<\xi_1<\dots <\xi_n<1$, $0<\eta_1<\dots <\eta_n<1$, and let
the alternation condition
\begin{align}
\label{eta_jkr} \eta_{j+k-r} <\xi_j <\eta_{j+k}
\end{align}
hold. We set $\xi_0:=\eta_0:=0$, $\xi_{n+1}:=\eta_{n+1}:=1$,
$\xi=(\xi_i)_{i=1}^n$, $\eta=(\eta_i)_{i=1}^n$. Let
$\xi_{i-1}<t<\xi_i$, $\eta_{j-1}<\tau<\eta_j$. We set
\begin{align}
\label{alpha} \alpha _l=\xi_l, \quad 1\le l\le i-1, \quad \alpha
_i=t, \quad \alpha_l =\xi_{l-1}, \quad i+1\le l\le n+1;
\end{align}
\begin{align}
\label{beta} \beta_l=\eta_l, \quad 1\le l\le j-1, \quad
\beta_j=\tau, \quad \beta_l=\eta_{l-1}, \quad j+1\le l\le n+1.
\end{align}
Then $\alpha_1<\alpha_2<\dots<\alpha_{n+1}$,
$\beta_1<\beta_2<\dots <\beta_{n+1}$.

Applying Lemma \ref{lem_kg0} and (\ref{gtt_xi}), we obtain the
following assertion.
\begin{Sta}
\label{a_b_perem} If
\begin{align}
\label{ab_f_per} \beta_{l+k-r}<\alpha_l<\beta_{l+k}
\end{align}
for any $l$, then $G(t, \, \tau, \, \xi, \, \eta)\ne 0$.
\end{Sta}

For $1\le i\le n+1$, $1\le j\le n+1$ we set $\Delta_{i,j}=
(\xi_{i-1}, \, \xi_i) \cap (\eta _{j-1}, \, \eta_j)$.

\begin{Sta}
Let $\Delta_{i,j} \ne \varnothing$. Then
\begin{align}
\label{ijkr} i+k-r \le j\le i+k.
\end{align}
\end{Sta}
\begin{proof}
We have $1\le i\le n+1$, $1\le j\le n+1$.

If $\Delta_{i,j} \ne \varnothing$, then either $\xi_{i-1}\le
\eta_{j-1}<\xi_i$ or $\eta_{j-1}<\xi_{i-1}<\eta_j$.

Consider the first case. If $i=1$, then $i+k-r<j$. If $i=n+1$,
then $j<i+k$. For $i\ge 2$ the inequality $\eta _{i-1+k-r}
\stackrel{(\ref{eta_jkr})}{<} \xi_{i-1}$ holds; hence, $\eta
_{i-1+k-r} < \eta _{j-1}$ and $j>i+k-r$. For $i\le n$ the
inequality $\xi_i\stackrel{(\ref{eta_jkr})}{<}\eta_{i+k}$ holds;
therefore, $\eta_{j-1}< \xi_i<\eta_{i+k}$ and $j<i+k+1$. Thus,
\begin{align}
\label{ineq1} i+k-r<j\le i+k.
\end{align}

Let us consider the second case. If $j=1$, then $j<i+k$. If
$j=n+1$, then $i-r+k<j$. For $j\ge 2$ we have
$\xi_{j-1-k}\stackrel{(\ref{eta_jkr})}{<}\eta_{j-1}$; therefore,
$\xi_{j-1-k}<\eta_{j-1}<\xi_{i-1}$ and $i>j-k$. For $j\le n$ the
inequality $\eta_j\stackrel{(\ref{eta_jkr})}{<}\xi _{j+r-k}$
holds; hence, $\xi_{i-1}<\eta_j <\xi_{j+r-k}$ and $j>i-1-r+k$.
Thus,
\begin{align}
\label{ineq2} i-r+k\le j<i+k.
\end{align}

The union of inequalities (\ref{ineq1}) and (\ref{ineq2}) gives
(\ref{ijkr}).
\end{proof}
\begin{Lem}
\label{posit1} Let
\begin{align}
\label{dijvrn} \Delta_{i,j}\ne \varnothing.
\end{align}
\begin{enumerate}
\item If $i-r+k+1\le j\le i+k-1$, then for any $\tau\in \Delta
_{i,j}$, $t\in \Delta _{i,j}$ we have $G(t, \, \tau, \, \xi, \,
\eta) \ne 0$.
\item If $j=i-r+k$, then for any $\tau \in \Delta_{i,j}$, $t\in
\Delta_{i,j}$ such that $t>\tau$ we have $G(t, \, \tau, \, \xi, \,
\eta) \ne 0$.
\item If $j=i+k$, then for any $\tau \in \Delta_{i,j}$, $t\in
\Delta_{i,j}$ such that $t<\tau$, we have $G(t, \, \tau, \, \xi,
\, \eta) \ne 0$.
\end{enumerate}
\end{Lem}
\begin{proof}
In all cases
\begin{align}
\label{tdij} t\in \Delta _{i,j}, \quad \tau\in \Delta_{i,j};
\end{align}
therefore,
$$
\xi_0<\xi_1<\dots<\xi_{i-1}<t<\xi_i<\dots<\xi_n<\xi_{n+1},
$$
$$
\eta_0<\eta_1<\dots<\eta_{j-1}<\tau<\eta_j<\dots<\eta_n<\eta_{n+1}.
$$
We define the points $\alpha_l$ and $\beta_l$ by formulas
(\ref{alpha}), (\ref{beta}). Let us prove that the alternation
conditions $\beta_{l+k-r}<\alpha_l<\beta_{l+k}$ hold and apply
Proposition \ref{a_b_perem}.

First we prove that $\beta_{l+k-r}<\alpha_l$.
\begin{enumerate}
\item If $l\le i-1$, then $\alpha_l\stackrel{(\ref{alpha})}{=}\xi_l$.
Further, $l+k-r \le i-1+k-r \stackrel{(\ref{ijkr})}{\le} j-1$.
Hence, $\beta_{l+k-r} \stackrel{(\ref{beta})}{=} \eta_{l+k-r}$. It
remains to apply (\ref{eta_jkr}).
\item If $l=i$, then $\alpha_l \stackrel{(\ref{alpha})}{=} t$.
By (\ref{ijkr}), $j\ge i+k-r$.
\begin{itemize}
\item If $i+k-r\le j-1$, then $\beta_{i+k-r}\stackrel
{(\ref{beta})}{=} \eta_{i+k-r}$; hence, it is sufficient to check
the inequality $\eta_{i+k-r}<t$. It follows from relations
$t\stackrel{(\ref{tdij})}{>}\eta_{j-1}\ge \eta_{i+k-r}$.
\item If $j=i-r+k$, then $\beta_{i+k-r} \stackrel{(\ref{beta})}
{=}\tau$. By conditions of Lemma (see assertion 2), we have
$t>\tau$.
\end{itemize}
\item If $l\ge i+1$, then $\alpha_l \stackrel{(\ref{alpha})}{=}\xi_{l-1}$.
\begin{itemize}
\item Let $l+k-r\le j-1$. Then $\beta_{l+k-r} \stackrel{(\ref{beta})}{=}\eta_{l+k-r}$.
The desired inequality follows from relations $\beta_{l+k-r}
=\eta_{l+k-r} \le \eta_{j-1} \stackrel{(\ref{dijvrn})}{<} \xi_i
\le \xi_{l-1} =\alpha_l$.
\item Let $l+k-r=j$. Then $\beta_{l+k-r} \stackrel{(\ref{beta})}{=}\tau$. The desired
inequality follows from relations $\beta_{l+k-r} =\tau
\stackrel{(\ref{tdij})}{<} \xi_i\le \xi_{l-1}=\alpha_l$.
\item Let $l+k-r\ge j+1$. Then $\beta_{l+k-r}
\stackrel{(\ref{beta})}{=}\eta_{l+k-r-1}$. It remains to apply
(\ref{eta_jkr}).
\end{itemize}
\end{enumerate}
Now we prove that $\alpha_l<\beta_{l+k}$.
\begin{enumerate}
\item If $l\le i-1$, then $\alpha_l \stackrel{(\ref{alpha})}{=} \xi_l$.
\begin{itemize}
\item Let $l+k\le j-1$. Then $\beta_{l+k} \stackrel{(\ref{beta})}{=} \eta_{l+k}$.
The desired inequality follows from (\ref{eta_jkr}).
\item Let $l+k=j$. Then $\beta_{l+k} \stackrel{(\ref{beta})}{=} \tau$. The desired inequality
follows from relations $\alpha_l=\xi_l\le \xi_{i-1}
\stackrel{(\ref{tdij})}{<} \tau=\beta_{l+k}$.
\item Let $l+k\ge j+1$. Then $\beta_{l+k} \stackrel{(\ref{beta})}{=} \eta_{l+k-1}$.
The desired inequality follows from relations $\alpha_l=\xi_l\le
\xi_{i-1} \stackrel{(\ref{dijvrn})}{<} \eta_j\le
\eta_{l+k-1}=\beta_{l+k}$.
\end{itemize}
\item If $l=i$, then $\alpha_l \stackrel{(\ref{alpha})}{=} t$. By (\ref{ijkr}), $i+k\ge j$.
\begin{itemize}
\item Let $i+k>j$. Then $\beta_{i+k}\stackrel{(\ref{beta})}{=}\eta_{i+k-1}\ge \eta_j
\stackrel{(\ref{tdij})}{>} t=\alpha_i$.
\item Let $i+k=j$. Then $\beta_{i+k} \stackrel{(\ref{beta})}{=} \tau$,
and the inequality $t<\tau$ follows from the condition of Lemma
(see assertion 3).
\end{itemize}
\item Let $l\ge i+1$. Then $\alpha_l \stackrel{(\ref{alpha})}{=} \xi_{l-1}$. Since $l+k>i+k
\stackrel{(\ref{ijkr})}{\ge} j$, we have
$\beta_{l+k}\stackrel{(\ref{beta})}{=} \eta_{l-k-1}$. Hence, the
desired inequality follows from (\ref{eta_jkr}).
\end{enumerate}
This completes the proof.
\end{proof}
\begin{Lem}
\label{posit2} Let $\Delta_{i,j} \ne \varnothing$, $\xi_{i-1}\ne
\eta_{j-1}$, $\tau _*= \max \{\xi_{i-1}, \, \eta_{j-1}\}$. Then
there exist $t_*\in (0, \, 1)$ and $\delta>0$ such that $G(t, \,
\tau, \, \xi, \, \eta) \ne 0$ for a.e. $(t, \, \tau) \in
(t_*-\delta, \, t_*+\delta) \times (\tau_* -\delta, \,
\tau_*+\delta)$.
\end{Lem}
\begin{proof}
{\it Case $\xi_{i-1}>\eta_{j-1}$}. Then $\tau_*=\xi_{i-1}$ and
$\xi_{i-1}<\eta_j$. We show that $i\ge j-k+1$. Indeed, otherwise
$i\stackrel{(\ref{ijkr})}{=}j-k$,
$\eta_{j-1}<\xi_{i-1}=\xi_{j-k-1}$; this contradicts with
(\ref{eta_jkr}).
\begin{enumerate}
\item Let $i\ge j-k+2$. We take $t_*\in \Delta_{i,j}$ and
sufficiently small $\delta >0$. Then $\tau_*<t_*$. If $\tau \in
(\tau_*, \, \tau_*+\delta)$, $|t-t_*|<\delta$, then $\tau<t$, $t$,
$\tau\in \Delta_{i,j}$, and by Lemma \ref{posit1} we get $G(t, \,
\tau, \, \xi, \, \eta) \ne 0$. Let $\tau\in (\tau_*-\delta, \,
\tau_*)$. Since $\xi_{i-1}>\eta_{j-1}$, the points $\alpha_l$ and
$\beta_l$ are defined by formulas (\ref{alpha}), (\ref{beta}), as
in the case $\tau \in (\tau_*, \, \tau_*+\delta)$. Hence, it
suffices to check that $\tau> \alpha _{j-k}$ and to apply
Proposition \ref{a_b_perem}. Indeed, since $j-k\le i-2$, we have
$\alpha_{j-k} \stackrel{(\ref{alpha})}{=} \xi_{j-k}\le
\xi_{i-2}<\tau$.
\item Let $i=j-k+1$. We take sufficiently small $\delta>0$ and $t_*\in \Delta
_{i-1,j}$. If $\tau \in (\tau_*-\delta, \, \tau_*)$, then for
small $\delta>0$ we have $\tau >t$. Since $i=j-k+1$, the case
$j=i-1-r+k$ is impossible. Hence, $G(t, \, \tau, \, \xi, \, \eta)
\ne 0$ by Lemma \ref{posit1} with $i:=i-1$, $j:=j$. Let $\tau \in
(\tau_*, \, \tau_*+\delta)$. Then points $\alpha_l$, $\beta_l$ are
defined as in case $\tau \in (\tau_*-\delta, \, \tau_*)$:
$\alpha_l=\xi_l$ for $l\le i-2$, $\alpha_{i-1}=t$,
$\alpha_l=\xi_{l-1}$ for $l\ge i$, $\beta_l$ is defined by
(\ref{beta}). Thus, in order to check the conditions
(\ref{ab_f_per}) it is sufficient to prove that $\tau<\alpha
_{j+r-k}$. Indeed, $j+r-k= i+r-1\ge i+1$; therefore,
$\alpha_{j+r-k} =\xi_{j+r-k-1}\ge \xi_i>\tau$.
\end{enumerate}

{\it Case $\xi_{i-1}<\eta_{j-1}$.} Then $i\le j+r-k-1$ (otherwise,
$i \stackrel{(\ref{ijkr})}{=} j+r-k$, $\xi_{i-1}<\eta_{j-1}=
\eta_{i-r+k-1}$, which contradicts with (\ref{eta_jkr})).
Moreover,
\begin{align}
\label{xi_i_eta_j} \xi_i>\eta_{j-1}, \quad \tau_*= \eta_{j-1}.
\end{align}
\begin{enumerate}
\item Let $i\le j+r-k-2$. We take $t_*\in \Delta_{i,j-1}$.
Then $\tau_*=\eta_{j-1}>t_*$. If $|t-t_*|<\delta$, $\tau\in
(\tau_*-\delta, \, \tau_*)$, then for small $\delta>0$ the
inequality $t<\tau$ holds and $\tau \in \Delta_{i,j-1}$. Here
$\alpha_l$ are defined by (\ref{alpha}), $\beta_l=\eta_l$ for
$l\le j-2$, $\beta_{j-1}=\tau$, $\beta_l=\eta_{l-1}$ for $l\ge j$.
By Lemma \ref{posit1} with $i:=i$, $j:=j-1$, $G(t, \, \tau, \,
\xi, \, \eta)\ne 0$. Let $\tau \in (\tau_*, \, \tau_*+\delta)$.
Then the numbers $\alpha_l$ are defined by (\ref{alpha}) as in the
case $\tau\in (\tau_*-\delta, \, \tau_*)$. The numbers $\beta_l$
are defined as follows: $\beta_l=\eta_l$ for $l\le j-1$,
$\beta_j=\tau$, $\beta_l=\eta_{l-1}$ for $l\ge i+1$. Thus, in
definition of $\beta_l$ we change $\beta_{j-1}$ and $\beta_j$:
$\beta_{j-1}=\eta_{j-1}$, $\beta_j=\tau$. Hence, in order to prove
(\ref{ab_f_per}) it is sufficient to show that
\begin{align}
\label{a1} \alpha_{j-1-k} <\eta_{j-1}< \alpha_{j-1+r-k},
\end{align}
\begin{align}
\label{a2} \alpha_{j-k}<\tau< \alpha_{j+r-k}.
\end{align}
Since $t_*\in \Delta_{i,j-1}$, we have $\xi_{i-1}<t<\xi_i$. By
conditions of Lemma, $\Delta_{i,j}\ne \varnothing$. Therefore, by
(\ref{ijkr}), $j-1-k\le i-1$. Hence,
$\alpha_{j-1-k}\stackrel{(\ref{alpha})}{=} \xi_{j-1-k}
<\eta_{j-1}$ by (\ref{eta_jkr}). Further, since $i\le j+r-k-2$, we
have $j-1+r-k\ge i+1$ and $\alpha_{j-1+r-k}
\stackrel{(\ref{alpha})}{=} \xi_{j-2+r-k}\ge \xi_i\stackrel
{(\ref{xi_i_eta_j})}{>}\eta_{j-1}$. This completes the proof of
(\ref{a1}). Let us check (\ref{a2}). We have $\alpha_{j+r-k}=\xi
_{j+r-k-1}\ge \xi_{i+1}>\tau$. In order to prove that $\tau>
\alpha_{j-k}$ it is sufficient to show that $\eta_{j-1}
>\alpha_{j-k}$ (see the second inequality of (\ref{xi_i_eta_j})).
If $j-k\le i-1$, then $\alpha_{j-k} \stackrel{(\ref{alpha})}{=}
\xi_{j-k} \le \xi_{i-1} <\eta_{j-1}$. If $j-k=i$, then
$\alpha_{j-k}=\alpha_i \stackrel{(\ref{alpha})}{=} t<\eta_{j-1}$
(since $t_*\in \Delta_{i,j-1}$ and $\delta$ is sufficiently
small).

\item Let $i=j+r-k-1$. We take $t_* \in \Delta_{i,j}$.
Then $t_*>\tau_*$. If $|t -t_*|< \delta$, $\tau \in (\tau_*, \,
\tau_*+\delta)$, then $\tau<t$ for small $\delta>0$. The case
$i+k=j$ is impossible, since $r\ge 2$. Therefore, $G(t, \, \tau,
\, \xi, \, \eta)\ne 0$ by Lemma \ref{posit1}. Let $\tau \in
(\tau_*-\delta, \, \tau_*)$. Then $\beta_{j-1}=\tau$,
$\beta_j=\eta_{j-1}$, other numbers $\alpha_l$ and $\beta_l$ are
the same as for $\tau \in (\tau_*, \, \tau_*+\delta)$; i.e., they
are defined by (\ref{alpha}) and (\ref{beta}). Hence, in order to
check the condition (\ref{ab_f_per}), it is sufficient to prove
that
\begin{align}
\label{a3} \alpha _{j-k-1}<\tau<\alpha_{j-1+r-k},
\end{align}
\begin{align}
\label{a4} \alpha _{j-k}<\eta_{j-1}<\alpha_{j+r-k}.
\end{align}
We have $j-k-1=i-r\le i-1$; therefore, $\alpha _{j-k-1}
\stackrel{(\ref{alpha})}{=} \xi_{j-k-1}\le \xi_{i-1}<\tau$ (the
last inequality holds since $\xi_{i-1}<\eta_{j-1}$, and $\tau$ is
close to $\tau_*\stackrel{(\ref{xi_i_eta_j})}{=} \eta_{j-1}$).
Further, $j-1+r-k=i$; hence, $\alpha_{j-1+r-k}
=\alpha_i\stackrel{(\ref{alpha})}{=} t>\tau$. This completes the
proof of (\ref{a3}). Let us check (\ref{a4}). We have
$\alpha_{j+r-k}=\alpha_{i+1} \stackrel{(\ref{alpha})}{=} \xi_i
\stackrel{(\ref{xi_i_eta_j})} {>} \eta_{j-1}$. Since $r\ge 2$, we
get $j-k=i-r+1\le i-1$, and
$\alpha_{j-k}\stackrel{(\ref{alpha})}{=}\xi_{j-k}\le
\xi_{i-1}<\eta_{j-1}$.
\end{enumerate}
This completes the proof.
\end{proof}
\begin{Lem}
\label{posit3} Let $r\ge 3$, $\Delta_{i,j}\ne \varnothing$,
$\tau_*=\xi_{i-1}=\eta_{j-1}$. Then there exist $\delta>0$ and
$t_*\in (0, \, 1)$ such that $G(t, \, \tau, \, \xi, \, \eta) \ne
0$ for a.e. $(t, \, \tau) \in (t_*-\delta, \, t_*+\delta) \times
(\tau_*-\delta, \, \tau_*+\delta)$.
\end{Lem}
\begin{proof}
First we prove that $i-1+k\ne j$ or $i\ne j-1+r-k$. Indeed,
otherwise $i=j-k+1$, $i=j-k-1+r$; i.e., $1=r-1$. It contradicts
with condition $r\ge 3$.

We show that
\begin{align}
\label{iknj} i+k\ne j, \quad i-r+k\ne j.
\end{align}
Indeed, if $i+k=j$ or $i=j+r-k$, then $\xi_{i-1}=\eta_{i+k-1}$ or
$\xi_{i-1}=\eta_{i-r+k-1}$. This contradicts with the alternation
condition (\ref{eta_jkr}).

{\it Case $i\ne j+1-k$}. We take $t_*\in \Delta _{i,j}$.

Let $\tau\in (\tau_*,\, \tau_*+\delta)$, $|t-t_*|<\delta$. Then
$t\in \Delta_{i,j}$, $\tau \in \Delta_{i,j}$ for small $\delta>0$.
Since $i+k\ne j$, $i-r+k\ne j$, we get $G(t, \, \tau, \, \xi, \,
\eta)\ne 0$ by Lemma \ref{posit1}.

Let $\tau \in (\tau_*-\delta, \, \tau_*)$. Then
\begin{align}
\label{alpha1} \alpha_l=\xi_l, \quad l\le i-1, \quad \alpha_i=t,
\quad \alpha_l=\xi_{l-1}, \quad l\ge i+1,
\end{align}
\begin{align}
\label{beta1} \beta_l=\eta_l, \quad l\le j-2, \quad
\beta_{j-1}=\tau, \quad \beta_l=\eta_{l-1}, \quad l\ge j.
\end{align}

Let us prove that
\begin{align}
\label{bet_lkr} \beta_{l+k-r}<\alpha_l.
\end{align}
\begin{enumerate}
\item If $l\le i-1$, then $\alpha_l\stackrel{(\ref{alpha1})}{=}\xi_l$, %. Let $l\le i-2$.
%Then $l+k-r \le i-2+k-r \stackrel{(\ref{ijkr})}{\le} j-2$,
%$\beta_l \stackrel{(\ref{beta1})}{=}\eta_l$ and (\ref{bet_lkr})
%follows from (\ref{eta_jkr}). If $l=i-1$, then
$l+k-r\le i-1+k-r \stackrel{(\ref{ijkr})}{\le} j-1$, and
$\beta_{l+k-r}\stackrel{(\ref{beta1})}{\le} \tau
<\xi_{i-1}=\xi_l$.
\item If $l=i$, then $\alpha_l\stackrel{(\ref{alpha1})}{=}t$. Further, $l+k-r=i+k-r
\stackrel{(\ref{ijkr})}{\le} j$; hence, $\beta_{l+k-r}\le \beta _j
\stackrel{(\ref{beta1})}{=}\eta_{j-1}$. It remains to apply the
inequality $t>\eta_{j-1}$ (it holds since $t_*\in \Delta_{i,j}$
and $\delta$ is sufficiently small).
\item If $l\ge i+1$, then $\alpha_l\stackrel{(\ref{alpha1})}{=}\xi_{l-1}$. Let $\beta_{l+k-r}\le
\tau$. Then (\ref{bet_lkr}) follows from inequalities
$\tau<\xi_i\le \xi_{l-1}$. If $\beta_{l+k-r}>\tau$, then
$\beta_{l+k-r} \stackrel{(\ref{beta1})}{=}\eta_{l+k-r-1}$, and
(\ref{bet_lkr}) follows from (\ref{eta_jkr}).
\end{enumerate}

Now we prove that
\begin{align}
\label{alp_l} \alpha_l<\beta _{l+k}.
\end{align}
\begin{enumerate}
\item If $l\le i-1$, then $\alpha_l\stackrel{(\ref{alpha1})}{=} \xi_l$.
\begin{itemize}
\item If $l+k\le j-2$, then
$\beta_{l+k}\stackrel{(\ref{beta1})}{=}\eta_{l+k}$, and
(\ref{alp_l}) follows from (\ref{eta_jkr}).
\item If $l+k=j-1$, then
$\beta_{l+k}\stackrel{(\ref{beta1})}{=} \tau>\xi_{i-2}$ (the last
inequality holds for small $\delta>0$, since $\tau_*=\xi_{i-1}$
and $\tau>\tau_*-\delta$). Hence, if $l+k=j-1$ and $l\le i-2$,
then $\beta_{l+k}>\xi_{i-2}\ge \xi_l=\alpha_l$. Let $l+k=j-1$,
$l=i-1$. Then $\xi_l=\xi_{i-1}=\eta_{j-1}=\eta_{l+k}$; this
contradicts with (\ref{eta_jkr}).
\item If $l+k\ge j$, then
$\beta_{l+k} =\eta_{l+k-1}$. Hence, (\ref{alp_l}) is equivalent to
$\xi_l<\eta_{l+k-1}$. We have $\xi_l\le \xi_{i-1}=\eta_{j-1}\le
\eta_{l+k-1}$; moreover, $\xi_l=\eta_{l+k-1}$ holds only for
$l=i-1$, $l+k=j$. Then $i=j+1-k$, which contradicts with our
assumption.
\end{itemize}
\item If $l=i$, then $\alpha_l\stackrel{(\ref{alpha1})}{=}t \in (\eta_{j-1}, \, \eta_j)$.
Further, $l+k=i+k \stackrel{(\ref{ijkr})}{\ge} j$; thus,
$\beta_{l+k}\stackrel{(\ref{beta1})}{=} \eta_{i+k-1}$. By
(\ref{ijkr}) and (\ref{iknj}), we have $i+k-1\ge j$; therefore,
$\beta_{l+k}\ge \eta_j>t=\alpha_l$.
\item If $l\ge i+1$, then $\alpha_l\stackrel{(\ref{alpha1})}{=}\xi_{l-1}$, $l+k\ge i+1+k
\stackrel{(\ref{ijkr})}{>} j$; hence,
$\beta_{l+k}\stackrel{(\ref{beta1})}{=}\eta_{l+k-1}$, and
(\ref{alp_l}) follows from (\ref{eta_jkr}).
\end{enumerate}
From (\ref{bet_lkr}), (\ref{alp_l}) and Proposition
\ref{a_b_perem} it follows that $G(t, \, \tau, \, \xi, \, \eta)\ne
0$.

{\it Case $i\ne j-1+r-k$}. We take $t_*\in \Delta_{i-1,j-1}$.

Let $\tau\in (\tau_*-\delta, \, \tau_*)$, $|t-t_*|<\delta$. Then
$t\in \Delta_{i-1,j-1}$, $\tau \in \Delta_{i-1,j-1}$. Since
$i-1+k\ne j-1$, $i-1-r+k\ne j-1$ by (\ref{iknj}), we have $G(t, \,
\tau, \, \xi, \, \eta)\ne 0$ by Lemma \ref{posit1}.

Let now $\tau \in (\tau_*, \, \tau_*+\delta)$, $\delta>0$ is
sufficiently small. Then
\begin{align}
\label{alpha2} \alpha_l=\xi_l, \quad l\le i-2, \quad
\alpha_{i-1}=t, \quad \alpha_l=\xi_{l-1}, \quad l\ge i,
\end{align}
\begin{align}
\label{beta2} \beta_l=\eta_l, \quad l\le j-1, \quad \beta_j=\tau,
\quad \beta_l=\eta_{l-1}, \quad l\ge j+1.
\end{align}

We prove that
\begin{align}
\label{alblk} \alpha_l< \beta_{l+k}.
\end{align}
\begin{enumerate}
\item If $l\le i-2$, then $\alpha_l\stackrel{(\ref{alpha2})}{=}\xi_l$.
In the case $l+k\le j-1$ we get $\beta_{l+k}
\stackrel{(\ref{beta2})}{=}\eta_{l+k}$, and (\ref{alblk}) follows
from (\ref{eta_jkr}). If $l+k=j$, then
$\beta_{l+k}\stackrel{(\ref{beta2})}{=} \tau>\xi_{i-1}>\xi_l$. If
$l+k\ge j+1$, then
$\beta_{l+k}\stackrel{(\ref{beta2})}{=}\eta_{l+k-1}\ge
\eta_j>t>\xi_{i-2}\ge \xi_l$.
\item If $l=i-1$, then $\alpha_l \stackrel{(\ref{alpha2})}{=} t$, $\beta_{l+k}=\beta_{i-1+k}$.
By (\ref{ijkr}), $j\le i+k$. If $j\le i-1+k$, then $\beta
_{i-1+k}\ge \beta _j \stackrel{(\ref{beta2})}{=} \tau>t$. Let
$j=i+k$. Then $\beta_{i-1+k}=\beta _{j-1}
\stackrel{(\ref{beta2})}{=}\eta_{j-1}>t$.
\item If $l\ge i$, then $\alpha_l\stackrel{(\ref{alpha2})}{=}\xi_{l-1}$. We have $l+k\ge i+k
\stackrel{(\ref{ijkr})}{\ge} j$. Let $l=i$. Then
$\alpha_l=\xi_{i-1}<\tau \stackrel{(\ref{beta2})}{=} \beta_j\le
\beta_{l+k}$. If $l\ge i+1$, then
$\alpha_l=\xi_{l-1}\stackrel{(\ref{eta_jkr})}{<} \eta_{l+k-1}
\stackrel{(\ref{beta2})}{=}\beta_{l+k}$ (the last equality holds
since $l+k\ge i+1+k \stackrel{(\ref{ijkr})}{\ge} j+1$).
\end{enumerate}

We prove that
\begin{align}
\label{blkr} \beta_{l+k-r}<\alpha_l.
\end{align}
\begin{enumerate}
\item If $l\le i-2$, then $\alpha_l\stackrel{(\ref{alpha2})}{=}\xi_l$. We show that $l+k-r\le
j-1$. Indeed, by (\ref{ijkr}), $i\le j+r-k$; hence, $l+k-r \le
i-2+k-r\le j-1$. Therefore, $\beta _{l+k-r}
\stackrel{(\ref{beta2})}{=}\eta_{l+k-r}$, and (\ref{blkr}) follows
from the alternation condition for $\xi, \, \eta$.
\item If $l=i-1$, then $\alpha_l\stackrel{(\ref{alpha2})}{=} t>\eta_{j-2}$.
Therefore, it is sufficient to prove that $\beta_{l+k-r}\le
\eta_{j-2}$. By (\ref{ijkr}) and (\ref{iknj}), $i\le j+r-k-1$;
hence, $l+k-r=i-1+k-r\le j-2$. Therefore, $\beta_{l+k-r}
\stackrel{(\ref{beta2})}{=}\eta_{l+k-r}\le \eta_{j-2}$.
\item If $l\ge i$, then $\alpha_l \stackrel{(\ref{alpha2})}{=} \xi_{l-1}$.
\begin{itemize}
\item Let $l+k-r\ge j+1$. Then $\beta_{l+k-r}
\stackrel{(\ref{beta2})}{=} \eta_{l+k-r-1}$, and (\ref{blkr})
follows from the alternation condition for $\xi, \, \eta$.
\item If $l+k-r=j$, then $\beta_{l+k-r} \stackrel{(\ref{beta2})}{=}
\tau<\xi_i$. Hence, it is sufficient to check that $\xi_{l-1}\ge
\xi_i$, which means that $l-1\ge i$. Suppose the contrary: let
$l=i$. Then $i=j+r-k$, which contradicts with (\ref{iknj}).
\item Let $l+k-r\le j-1$. Then
$\beta_{l+k-r}\stackrel{(\ref{beta2})}{=}\eta_{l+k-r}$, and
(\ref{blkr}) is equivlent to the inequality
$\xi_{l-1}>\eta_{l+k-r}$. We have $\xi_{l-1}\ge
\xi_{i-1}=\eta_{j-1}\ge \eta_{l+k-r}$; in addition,
$\xi_{l-1}=\eta_{l+k-r}$ only if $l=i$, $l+k-r=j-1$. Then
$i=j-1+r-k$, which contradicts with our assumption.
\end{itemize}
\end{enumerate}

From (\ref{alblk}), (\ref{blkr}) and Proposition \ref{a_b_perem}
it follows that $G(t, \, \tau, \, \xi, \, \eta)\ne 0$. This
completes the proof.
\end{proof}
\begin{Cor}
\label{r32l} Let $\tau_*\in (0, \, 1)$. Suppose that $r\ge 3$, or
$r=2$ and $\tau_*\notin\{\eta_j:\; 1\le j\le n, \;
\eta_j=\xi_j\}$. Then there exist $\delta>0$ and $t_*\in (0, \,
1)$ such that $G(t, \, \tau, \, \xi, \, \eta) \ne 0$ for a.e. $(t,
\, \tau) \in (t_*-\delta, \, t_*+\delta) \times (\tau_*-\delta, \,
\tau_*+\delta)$.
\end{Cor}

\section{Auxiliary assertions for $r=2$}

In this section we obtain analogues of Lemmas 10 and 11 from
\cite{vas_spn}. The main part of the proof is similar to arguments
from \cite[\S 5]{vas_spn}.

Since $r=2$, $1\le k\le r-1$, we have $k=1$. From (\ref{htt}) it
follows that
\begin{align}
\label{httt} H(t, \, \tau) =\min\{t, \, \tau\}.
\end{align}

Since the operator $\tilde I^k_{r,g,v}:L_p[0, \, 1] \rightarrow
L_q[0, \, 1]$ is bounded, from (\ref{gv}) it follows that
\begin{align}
\label{int_tvq} \int \limits _0^T v^q(t) t^q\, dt<\infty
\quad\text{for any } 0<T<1.
\end{align}

Let $(\overline{x}, \, \overline{y}, \overline{\theta}) \in SP_n$.
Then (see (\ref{eq_main}))
\begin{align}
\label{xy_spn2} \left\{\begin{array}{l} \ddot{\overline{x}}
=-g^{p'}\overline{y}_{(p')}, \quad \ddot{\overline{y}} =-
\overline{\theta}^q v^q
\overline{x}_{(q)}, \\
\overline{x}(0)=\dot{\overline{x}}(1) =\overline{y}(0)=
\dot{\overline{y}}(1)=0, \\ \left\|
\frac{\ddot{\overline{x}}}{g}\right\| _{L_p[0, \, 1]} =1,
\end{array} \right.
\end{align}
$0<\xi_1<\xi_2<\dots <\xi_n<1$ are points of sign change of the
function $\overline{x}$, $0<\eta_1<\eta_2<\dots<\eta_n<1$ are
points of sign change of the function $\overline{y}$. From the
Rolle's theorem and Proposition \ref{per_zn} it follows that
$\dot{\overline{x}}$ has exactly $n$ points of sign change.
Without loss of generality we may assume that $\overline{x}(t)>0$
in the right punctured neighborhood of zero. Hence,
$\dot{\overline{x}}(t)>0$ in the right punctured neighborhood of
zero. Then
\begin{enumerate}
\item if $n$ is even, then $\dot{\overline{x}}(t)>0$ and
$\ddot{\overline{x}}(t)\le 0$ in the left punctured neighborhood
of 1. By the first equation (\ref{xy_spn2}), $\overline{y}(t)>0$
in the left punctured neighborhood of 1. Since $n$ is even,
$\overline{y}(t)>0$ in the right punctured neighborhood of zero.
\item if $n$ is odd, then $\dot{\overline{x}}(t)<0$ and
$\ddot{\overline{x}}(t)\ge 0$ in the left punctured neighborhood
of 1. By the first equation (\ref{xy_spn2}), $\overline{y}(t)<0$
in the left punctured neighborhood of 1. Since $n$ is,
$\overline{y}(t)>0$ in the right punctured neighborhood of zero.
\end{enumerate}
Thus,
\begin{align}
\label{sign} \overline{x}(t)>0, \quad \overline{y}(t)>0 \quad
\text{in the right neighborhood of zero}.
\end{align}

Denote $\xi_0=\eta_0=0$, $\xi_{n+1}=\eta_{n+1}=1$.

Let
\begin{align}
\label{lambda_def} \Lambda:= \{l=\overline{1, \, n}:\;
\eta_i=\xi_i\} \ne \varnothing.
\end{align}
We denote the elements of $\Lambda$ by $l_s$, $1\le s\le m$. Also
we set $l_0=0$, $l_{m+1}=n+1$.

For $f\in W^{2,1}_{p,g}[0, \, 1]$ we denote
\begin{align}
\label{ql_f} Q_Lf(t) =\int \limits _0^1 G(t, \, \tau, \, \xi, \,
\eta) \ddot f(\tau)\, d\tau.
\end{align}
Then
\begin{align}
\label{qlxij0} Q_Lf(\xi_j)\stackrel{(\ref{il01g})}{=}0, \quad 1\le
j\le n,
\end{align}
\begin{align}
\label{pl_incl} P_Lf:=f+Q_Lf\stackrel{(\ref{xt1rk}), (\ref{ql_f})}
{\in} L_n:= \left\{ \sum \limits _{j=1}^n c_j h_j:\; c_j\in \R,
\quad 1\le j\le n\right\},
\end{align}
where
\begin{align}
\label{hjt} h_j(t)=H(t, \, \eta_j)\stackrel{(\ref{httt})}{=}\min
\{t, \, \eta_j\}.
\end{align}

Let $\psi _1(t) =\min \{t, \, \eta_1\}$. For $2\le j\le n$ we set
\begin{align}
\label{psi_j_def} \psi_j(t)=\left\{
\begin{array}{l} 0, \quad 0\le t\le \eta_{j-1},
\\ t-\eta_{j-1}, \quad \eta_{j-1}\le t\le \eta_j, \\
\eta_j-\eta_{j-1}, \quad \eta_j\le t\le 1.\end{array}\right.
\end{align}
Then
\begin{align}
\label{hj_sum} h_j=\sum \limits_{i=1}^j \psi_i;
\end{align}
hence,
\begin{align}
\label{l_n_psi_j} L_n={\rm span}\, \{\psi_j\}_{j=1}^n.
\end{align}

For $0\le s\le m-1$ we denote $$\eta^{(s)} = (\eta _{l_s+1}, \,
\eta _{l_s+2}, \, \dots, \, \eta _{l_{s+1}}), \quad \xi^{(s)} =
(\xi _{l_s+1}, \, \xi _{l_s+2}, \, \dots, \, \xi _{l_{s+1}}).$$ In
addition, we set $\eta^{(m)} = (\eta _{l_m+1}, \, \eta _{l_m+2},
\, \dots, \, \eta _n)$, $\xi^{(m)} = (\xi _{l_m+1}, \, \xi
_{l_m+2}, \, \dots, \, \xi _n)$. For $f\in W^{2,1}_{p,g}[0, \,
1]$, $0\le s\le m$ we set
$$
Q^0_s f(t)= \int \limits _{\eta_{l_s}} ^{\eta_{l_{s+1}}} G(t, \,
\tau, \, \xi ^{(s)}, \, \eta^{(s)}) \ddot f(\tau)\, d\tau.
$$

The following assertion is the analogue of Lemma 5 from
\cite{vas_spn}.

\begin{Lem}
\label{lem_ogr} The equality $Q_Lf|_{[\eta_{l_s}, \,
\eta_{l_{s+1}}]} =Q^0_sf$ holds.
\end{Lem}
\begin{proof}
Let $0\le s\le m-1$. By direct calculations it can be checked that
\begin{align}
\label{q0s} \int \limits _{\eta_{l_s}}^{\eta_{l_{s+1}}} \min
\{t-\eta_{l_s}, \, \tau-\eta_{l_s}\}\ddot f(\tau)\, d\tau=
-f(t)+f(\eta_{l_s})+(t-\eta_{l_s}) \dot f(\eta_{l_{s+1}}).
\end{align}
In addition, $$(-f(t)+f(\eta_{l_s})+(t-\eta_{l_s}) \dot
f(\eta_{l_{s+1}}))|_{t=\eta_{l_s}}=0,$$
$$\left.\frac{d}{dt}(-f(t)+f(\eta_{l_s})+(t-\eta_{l_s}) \dot
f(\eta_{l_{s+1}})) \right|_{t=\eta_{l_{s+1}}} =0.$$ Therefore,
\begin{align}
\label{il_e_ls} \int \limits _{\eta_{l_s}} ^{\eta_{l_{s+1}}}
\min\{t-\eta_{l_s}, \, \tau -\eta_{l_s}\} \ddot f(\tau)\, d\tau
-Q^0_sf(t) \in {\rm span} \{\psi_j\}_{j=l_s+1} ^{l_{s+1}}
\end{align}
(it is the analogue of formulas (\ref{pl_incl}) and
(\ref{l_n_psi_j})). Since for any $t\in [\eta_{l_s}, \,
\eta_{l_{s+1}}]$ we have $t-\eta_{l_s} = \sum \limits
_{j=l_s+1}^{l_{s+1}} \psi_j(t)$, by (\ref{q0s}) and
(\ref{il_e_ls}) there are numbers $\{b_j\}_{j=l_k+1} ^{l_{s+1}}$
such that
\begin{align}
\label{ft} f(t)-f(\eta_{l_s})+Q^0_sf(t)=\sum \limits_{j=l_s+1}
^{l_{s+1}} b_j\psi_j(t), \quad \eta_{l_s}\le t\le \eta_{l_{s+1}}.
\end{align}

By (\ref{pl_incl}) and (\ref{hj_sum}), there are numbers
$\{c_j\}_{j=1} ^n$ such that $f(t)+Q_Lf(t)=\sum \limits _{j=1}^n
c_j\psi_j$. For $j\ge l_{s+1}+1$ we have $\psi_j|_{[\eta_{l_s}, \,
\eta_{l_{s+1}}]} \stackrel{(\ref{psi_j_def})}{\equiv} 0$, for
$j\le l_s$ we have $\psi_j|_{[\eta_{l_s}, \, \eta_{l_{s+1}}]}
\equiv {\rm const}$. Hence, there is $A\in \R$ such that
$f(t)+Q_Lf(t)=A+\sum \limits _{j=l_s+1}^{l_{s+1}} c_j\psi_j$.
Since $Q_Lf(\eta_{l_s})\stackrel{(\ref{lambda_def})}{=}
Q_Lf(\xi_{l_s})\stackrel{(\ref{qlxij0})}{=} 0$ and
$\psi_j(\eta_{l_s})\stackrel{(\ref{psi_j_def})}{=} 0$, $l_s+1\le
j\le l_{s+1}$, we get $A=f(\eta_{l_s})$ and
\begin{align}
\label{ft1} f(t)-f(\eta_{l_s})+Q_Lf(t)=\sum \limits_{j=l_s+1}
^{l_{s+1}} c_j\psi_j(t), \quad \eta_{l_s}\le t\le \eta_{l_{s+1}}.
\end{align}

Since $Q_Lf(\xi_i)\stackrel{(\ref{qlxij0})}{=}0$ and similarly
$Q^0_sf(\xi_i)=0$, $l_s+1\le i\le l_{s+1}$, we get from (\ref{ft})
and (\ref{ft1}) that
$$
\sum \limits_{j=l_s+1} ^{l_{s+1}} (c_j-b_j)\psi_j(\xi_i)=0, \quad
l_s+1\le i\le l_{s+1}.
$$
Let us prove that the matrix $(\psi_j(\xi_i))_{l_s+1\le i,\, j\le
l_{s+1}}$ is non-degenerate. By (\ref{hjt}) and (\ref{hj_sum}), it
holds if and only if the matrix $(\min(\xi_i, \,
\eta_j))_{l_s+1\le i,\, j\le l_{s+1}}$ is non-degenerate; the last
property follows from (\ref{httt}) and analogue of (\ref{chg0})
for $l_s+1\le i,\, j\le l_{s+1}$ (see Proposition \ref{peremez1}
and Lemma \ref{lem_kg0}).

Thus, $c_j=b_j$. This together with (\ref{ft}) and (\ref{ft1})
yields the assertion of Lemma for $s<m$.

Let $s=m$. Then from (\ref{q0s}), the condition $\dot f(1)=0$ and
the equality
$$
\int \limits _{\eta_{l_m}} ^1 \min\{t-\eta_{l_m}, \, \tau
-\eta_{l_m}\} \ddot f(\tau)\, d\tau -Q^0_mf(t) \in {\rm span}
\{\psi_j\}_{j=l_m+1} ^n
$$
it follows that
$$
f(t)-f(\eta_{l_m})+Q^0_mf(t)=\sum \limits_{j=l_m+1} ^n
b_j\psi_j(t), \quad \eta_{l_m}\le t\le 1.
$$
Similarly as formula (\ref{ft1}) it can be proved that
$$
f(t)-f(\eta_{l_m})+Q_Lf(t)=\sum \limits_{j=l_m+1} ^n c_j\psi_j(t),
\quad \eta_{l_m}\le t\le 1.
$$
After that we argue similarly as for $s<m$.
\end{proof}

For $0\le \alpha <\beta\le 1$ we set $\Phi _{[\alpha, \,
\beta]}(f) =\int \limits _\alpha ^\beta |vf|^q\, dt$. Denote
$$
\Phi_s(f) = \left\{ \begin{array}{l} \Phi_{[\eta_{l_s-1}, \, \eta
_{l_s}]}(f), \quad 1\le s\le m, \\ \Phi _{[\eta _{l_m}, \, \eta
_{l_m+1}]}(f), \quad s=m+1,\end{array} \right.
$$
\begin{align}
\label{vrph_s} \varphi_s=\psi_{l_s}, \quad 1\le s\le m,
\end{align}
\begin{align}
\label{vrph_m1} \varphi_{m+1} = \eta_{l_m+1}-\eta_{l_m}-
\psi_{l_m+1} \quad\text{for}\;l_m<n, \quad \varphi_{m+1}=1
\quad\text{for}\;l_m=n.
\end{align}

The following assertion can be proved by direct calculations (see,
e.g., \cite{fabian_hajek}, p. 268, Exercise 8.29).
\begin{Sta}
\label{diff} Let $\mu$ be a measure, $1<q<\infty$, $\Phi(f)=\int
\limits _0^1 |f(t)|^q\, d\mu(t)$. Then the function $\Phi$ has the
Lagrange variation, which is equal to
\begin{align}
\label{phi_f_h} \Phi'(f)[h] =q\int \limits_0^1 (f)_{(q)} h\, dt.
\end{align}
\end{Sta}

\begin{Rem}
This function is Fr\'{e}chet differentiable and even in a more
strong sense; it follows from the uniform smoothness of the space
$L_q$ for $1<q<\infty$ (see \cite{fabian_hajek}, \S 9).
\end{Rem}

The following lemma is similar to Proposition 6 from
\cite{vas_spn}.
\begin{Lem}
\label{lem_ineq} The following inequalities hold:
$$
\dot{\overline{x}}(\eta_{l_s}) \Phi_s'(\overline{x}) [\varphi_s]
<0, \quad 1\le s\le m, \quad \dot{\overline{x}}(\eta_{l_m})
\Phi'_{m+1}(\overline{x}) [\varphi_{m+1}] >0.
$$
\end{Lem}
\begin{proof}
We first consider the case $1\le s\le m$. We have
\begin{align}
\label{phi_pr_s} \Phi'_s(\overline{x})[\varphi _s]
\stackrel{(\ref{psi_j_def}), (\ref{vrph_s}), (\ref{phi_f_h})}{=} q
\int \limits _{\eta _{l_s-1}} ^{\eta_{l_s}} v^q(t)
\overline{x}_{(q)}(t) (t-\eta _{l_s-1}) \, dt
\stackrel{(\ref{xy_spn2})}{=} -q\overline{\theta}^{\, -q} \int
\limits _{\eta _{l_s-1}} ^{\eta_{l_s}} \ddot{\overline{y}}(t)
(t-\eta _{l_s-1})\, dt=:I_s.
\end{align}
If $l_s>1$, then $\dot{\overline{y}} \in AC[\eta_{l_s-1}, \,
\eta_{l_s}]$; we apply the formula of integration by parts and
obtain
$$
\int \limits _{\eta _{l_s-1}} ^{\eta_{l_s}} \ddot{\overline{y}}(t)
(t-\eta _{l_s-1})\, dt= (t-\eta
_{l_s-1})\dot{\overline{y}}(t)|_{\eta_{l_s-1}}^{\eta_{l_s}} -\int
\limits _{\eta _{l_s-1}} ^{\eta_{l_s}} \dot{\overline{y}}(t)\, dt=
(\eta_{l_s} -\eta_{l_s-1})\dot{\overline{y}}(\eta_{l_s})
$$
(here we apply the equalities $\overline{y}(\eta_j)=0$, $1\le j\le
n$). Hence, there exists $A_s>0$ such that
\begin{align}
\label{i_eq_adot} I_s=-A_s \dot{\overline{y}}(\eta_{l_s}), \quad
1\le s\le m, \quad l_s>1.
\end{align}
Let $s=1$ and $l_1=1$. Then $\eta_{l_1-1}=0$. Since
$\overline{x}\in L_{q,v}[0, \, 1]$, we have
$v^{q-1}\overline{x}_{(q)} \in L_{q'}[0, \, 1]$. This together
with (\ref{int_tvq}) and the H\"{o}lder inequality yields that the
function $v^q(t) \overline{x}_{(q)}(t) t$ is Lebesgue integrable.
By (\ref{xy_spn2}), the function $\ddot{\overline{y}}(t) t$ is
also Lebesgue integrable. From the formula of integration by parts
we get for small $\varepsilon
>0$
$$
\int \limits _{\varepsilon} ^{\eta_1} \ddot{\overline{y}}(t) t\,
dt = t\dot{\overline{y}}(t) |_\varepsilon ^{\eta_1} -\int \limits
_{\varepsilon}^{\eta_1}\dot{\overline{y}}(t)\, dt =-\varepsilon
\dot{\overline{y}}(\varepsilon)+\eta_1 \dot{\overline{y}}(\eta_1)
-\overline{y}(\eta_1)+\overline{y}(\varepsilon).
$$
The both parts of this equality have a limit as $\varepsilon \to
0+$; here $\overline{y}(\eta_1)=0$, $\overline{y}(\varepsilon)
\underset{\varepsilon \to 0+}{\to} 0$ by (\ref{xy_spn2}). Let
$\varepsilon\dot{\overline{y}}(\varepsilon) \underset{\varepsilon
\to 0+}{\to} c$. If $c\ne 0$, then for small $\delta>0$
$$
\int \limits _0^\delta |\dot{\overline{y}}|\, dt \ge \frac{|c|}{2}
\int \limits _0^\delta \frac{dt}{t}.
$$
Since $\overline{y}$ is absolutely continuous on $[0, \, 1)$, the
left-hand side has a finite limit. The right-hand side is
infinite, which leads to a contradiction. Hence, $c=0$ and $\int
\limits _{\varepsilon} ^{\eta_1} \ddot{\overline{y}}(t) t\, dt \to
\eta_1 \dot{\overline{y}}(\eta_1)$ as $\varepsilon \to +0$. Thus,
there exists $A_1>0$ such that
\begin{align}
\label{i_eq_adot1} I_1=-A_1 \dot{\overline{y}}(\eta_{l_1}), \quad
l_1=1.
\end{align}

Let now $s=m+1$. If $l_m<n$, then
$$
\Phi'_{m+1}(\overline{x})[\varphi _{m+1}]
\stackrel{(\ref{psi_j_def}), (\ref{vrph_m1}), (\ref{phi_f_h})}{=}
q \int \limits _{\eta _{l_m}} ^{\eta_{l_m+1}} v^q(t)
\overline{x}_{(q)}(t) (\eta_{l_m+1}-t) \, dt
\stackrel{(\ref{xy_spn2})}{=}
$$
$$
=-q\overline{\theta}^{\, -q} \int \limits _{\eta _{l_m}}
^{\eta_{l_m+1}} \ddot{\overline{y}}(t) (\eta_{l_m+1} -t)\,
dt=:I_{m+1}.
$$
Similarly as (\ref{i_eq_adot}) we prove that
\begin{align}
\label{i_eq_adot2} I_{m+1}=A_{m+1}\dot{\overline{y}}(\eta_{l_m}).
\end{align}
Let $l_m=n$. Then
$$
\Phi'_{m+1}(\overline{x})[\varphi _{m+1}]
\stackrel{(\ref{vrph_m1}), (\ref{phi_f_h})}{=} q \int \limits
_{\eta _n} ^1 v^q(t) \overline{x}_{(q)}(t)  \, dt
\stackrel{(\ref{xy_spn2})}{=} -q \overline{\theta}^{\, -q} \int
\limits _{\eta_n}^1 \ddot{\overline{y}}(t)\, dt=:I_{m+1}.
$$
Since $\dot{\overline{y}}$ is absolutely continuous in a left
neighborhood of 1 and $\dot{\overline{y}}(1)=0$, then
\begin{align}
\label{i_eq_adot3} I_{m+1}=q \overline{\theta}^{\, -q}
\dot{\overline{y}} (\eta_n).
\end{align}

From (\ref{i_eq_adot}), (\ref{i_eq_adot1}), (\ref{i_eq_adot2}),
(\ref{i_eq_adot3}) it follows that in order to complete the proof
it is sufficient to check that $\dot{\overline{x}}(\eta_{l_s})
\dot{\overline{y}}(\eta_{l_s})>0$, $1\le s\le m$ (recall that by
assertion 3 of Proposition 1 we have
$\dot{\overline{x}}(\eta_{l_s}) \dot{\overline{y}}(\eta_{l_s})\ne
0$).

Let $l_s$ be even. By (\ref{sign}), in the left neighborhood of
$\eta_{l_s}$ we have $\overline{x}(t)<0$, $\overline{y}(t)<0$, and
in the right neighborhood of $\eta_{l_s}$ we have
$\overline{x}(t)>0$, $\overline{y}(t)>0$. Hence,
$\dot{\overline{x}}(\eta_{l_s})>0$,
$\dot{\overline{y}}(\eta_{l_s})>0$. Similarly for odd $l_s$ we get
$\dot{\overline{x}}(\eta_{l_s})<0$,
$\dot{\overline{y}}(\eta_{l_s})<0$.
\end{proof}

Now we introduce notation from \cite{vas_spn}.

For ${\bf c}=(c_0, \, c_1, \, \dots, \, c_m) \in \R^{m+1}$ we
denote
$$
\psi_{{\bf c}} =\sum \limits _{s=0}^m c_s
\frac{\ddot{\overline{x}}}{g} \cdot \chi _{[\eta_{l_s}, \,
\eta_{l_{s+1}}]}.
$$
Let
\begin{align}
\label{c_c_in_rm} {\cal C}=\left\{{\bf c} \in \R^{m+1}:\; \|\psi
_{{\bf c}}\|^p_{L_p[0, \, 1]} \equiv\sum \limits _{s=0}^m |c_s|^p
\int \limits _{\eta_{l_s}} ^{\eta_{l_{s+1}}}
\left|\frac{\ddot{\overline{x}}} {g}\right|^p\, dt=1\right\}.
\end{align}
For $\delta>0$ we set
\begin{align}
\label{m_delta} M_\delta = \left\{f\in W^{2,1}_{p,g}[0, \, 1]:\;
\exists {\bf c}={\bf c}(f, \, \delta) \in {\cal C}:\; \left\|
\frac{\ddot{f}}{g} -\psi_{{\bf c}}\right\| _{L_p[0, \, 1]} \le
\delta \right\}.
\end{align}

Let $f\in W^{2,1}_{p,g}[0, \, 1]$. For $1\le s\le m-1$ we set
$$
Q_sf(t)=\left\{ \begin{array}{l} Q^0_sf(t), \quad \eta_{l_s}\le
t\le \eta_{l_{s+1}}, \\ (Q^0_sf)'(\eta_{l_s}) (t-\eta_{l_s}) +\int
\limits_t^{\eta_{l_s}} (\tau-t)\ddot{f}(\tau)\, d\tau, \quad 0\le
t\le \eta_{l_s}.\end{array} \right.
$$
For $0<\gamma_s<\eta_{l_s}-\eta_{l_s-1}$ we set
$$
\hat Q_sf(t)=Q_sf(t)-\frac{Q_sf(\eta_{l_s}-\gamma_s)}
{\eta_{l_s+1}-\eta_{l_s}+\gamma_s}(\eta_{l_s+1}-t)_+, \quad t\in
[\eta_{l_s}-\gamma_s, \, \quad \eta_{l_{s+1}}].
$$
Then $\hat Q_sf(\eta_{l_s}-\gamma_s)=0$.

For $l_s=l_{s+1}-1$ we denote
$$
\varphi _{s+1}^{\gamma_{s}}(t)= \left\{ \begin{array}{l} 0, \quad
0\le t\le \eta_{l_s}-\gamma_{s}, \\ t-\eta_{l_s} + \gamma_{s},
\quad \eta_{l_s}-\gamma_{s}\le t\le \eta_{l_{s+1}},
\\ \eta_{l_{s+1}} -\eta_{l_s} +\gamma_{s}, \quad \eta_{l_{s+1}}
\le t\le 1.
\end{array} \right.
$$

Let
$$
\tilde \varphi_{s+1}=\left\{ \begin{array}{l} \varphi_{s+1}, \quad
l_{s+1}-1>l_s, \\ \varphi _{s+1}^{\gamma_{s}}, \quad
l_{s+1}-1=l_s, \end{array} \right.
$$
$$\tilde Q_sf(t)=\hat Q_sf(t)+\rho_s\tilde \varphi_{s+1}(t), \quad t\in
[\eta_{l_s} -\gamma_s, \, \eta_{l_{s+1}} -\gamma_{s+1}], \quad
1\le s\le m-1,$$ $$\tilde Q_0f(t)=Q_0f(t)+\rho_0\varphi_1(t),
\quad t\in [0, \, \eta_{l_1} -\gamma_1],$$ $$\tilde
Q_mf(t)=Q_mf(t)+\rho_m \varphi_{m+1}(t), \quad t\in
[\eta_{l_m}+\gamma_{m+1}, \, 1];$$ here the numbers $\rho_s$ are
such that $\tilde Q_sf(\eta_{l_{s+1}} -\gamma _{s+1}) =0$, $0\le
s\le m-1$, $\tilde Q_mf(\eta_{l_m}+\gamma_{m+1})=0$. Denote
$Q^*f(t)=-f(t)+f(\eta_{l_m}-\gamma_m) +c(t-\eta_{l_m}+\gamma_m)$,
$t\in [\eta_{l_m}-\gamma_m, \, \eta_{l_m}+\gamma_{m+1}]$, where
$c\in \R$ is such that $Q^*f(\eta_{l_m}+\gamma_{m+1}) =0$;
$$
\tilde Qf(t)=\left\{ \begin{array}{l} \tilde Q_0f(t), \quad t\in
[0, \, \eta_{l_1}-\gamma_1], \\ \tilde Q_sf(t), \quad t\in
[\eta_{l_s} -\gamma_s, \, \eta_{l_{s+1}}-\gamma _{s+1}], \;\; 1\le
s\le m-1, \\ Q^*f(t), \quad t\in [\eta_{l_m}-\gamma_m, \,
\eta_{l_m}+\gamma_{m+1}], \\ \tilde Q_mf(t), \quad t\in
[\eta_{l_m}+\gamma_{m+1}, \, 1].\end{array}\right.
$$

We set $\tilde\eta_{l_s} =\eta_{l_s}-\gamma_s$ for $1\le s\le m$,
$\tilde \eta_j=\eta_j$ for $j\in \{1, \, \dots, \, l_m-1\}
\backslash \{l_1, \, \dots, \, l_{m-1}\}$, $\tilde \eta_{l_m+1}
=\eta_{l_m}+\gamma_{m+1}$, $\tilde \eta_j=\eta_{j-1}$, $j\in
\{l_m+2, \, \dots, \, n+1\}$,
\begin{align}
\label{til_psi_j} \tilde \psi_j(t)=\left\{ \begin{array}{l} 0,
\quad 0\le t\le \tilde \eta_{j-1}, \\ t-\tilde \eta_{j-1}, \quad
\tilde \eta_{j-1} \le t\le \tilde \eta_j, \\ \tilde \eta_j-\tilde
\eta_{j-1}, \quad \tilde \eta_j\le t\le 1,\end{array}\right.
\end{align}
$1\le j\le n+1$,
\begin{align}
\label{tilde_l} \tilde L=\left\{ \sum \limits _{j=1}^{n+1}
c_j\tilde \psi_j:\; c_j\in \R\right\}.
\end{align}

Arguing similarly as in \cite{vas_spn}, we obtain the following
assertions.
\begin{Lem}
\label{r2_lem0} The inclusion $f+\tilde Qf\in \tilde L$ holds.
\end{Lem}

\begin{Lem}
\label{r2_lem} There are numbers $\delta>0$, $\gamma_0>0$, and
$C_*=C_*(g, \, v, \, p, \, q, \, \overline{x}, \, \overline{y})$,
for which, for any $\gamma\in (0, \, \gamma_0)$, there are
$\gamma_1, \, \dots, \, \gamma_{m+1} \in (0, \, \gamma)$ such that
$$
\int \limits_0^1 v^q(t)|Q_Lf(t)|^q\, dt-\int \limits_0^1 v^q(t)
|\tilde Qf(t)|^q\, dt\ge C_*\gamma
$$
for any  $f\in M_\delta$.
\end{Lem}

\section{The strict decreasing of widths}

Let $n\in \Z_+$. We prove that
\begin{align}
\label{strict_monot} d_{n+1}(W^{r,k}_{p,g}[0, \, 1], \, L_{q,v}[0,
\, 1])< d_n(W^{r,k}_{p,g}[0, \, 1], \, L_{q,v}[0, \, 1]).
\end{align}

Let $\overline{\theta} = \overline{\theta}_n$, $(\overline{x}, \,
\overline{y}, \, \overline{\theta}) \in SP_n$.

For $r=2$ we define the set $\Lambda$ by (\ref{lambda_def}).

Consider two cases.
\begin{enumerate}
\item $r=2$, $\Lambda \ne \varnothing$,
\item either $r=2$ and $\Lambda=\varnothing$ or $r\ge 3$.
\end{enumerate}
In the first case we choose $\delta$ according to Lemma
\ref{r2_lem} and take $M=M_\delta$ as defined in (\ref{m_delta}).
In the second case we set
$$
M=\left\{ f\in W^{r,k}_{p,g}[0, \, 1]:\; \left\| \frac{f^{(r)}}{g}
-\frac{\overline{x}^{(r)}}{g}\right\|_{L_p[0, \, 1]}\le \delta
\quad \text{or} \quad \left\| \frac{f^{(r)}}{g}
+\frac{\overline{x}^{(r)}}{g}\right\|_{L_p[0, \, 1]}\le
\delta\right\}.
$$

For $r=2$ we define the function $P_Lf$ by formula
(\ref{pl_incl}). For $r\ge 3$ we set
$$
P_Lx(t):=x(t) - (-1)^{r-k}\int \limits_0^1 G(t, \, \tau, \, \xi,
\, \eta) x^{(r)}(\tau)\, d\tau=\sum \limits _{j=1}^n c_j(x) H(t,
\, \eta_j)
$$
(see (\ref{xt1rk})).

Let $N=W^{r,k}_{p,g}[0, \, 1] \backslash M$. We show that there
exists $\sigma>0$ such that
\begin{align}
\label{fplf} \|f-P_Lf\|_{L_{q,v}[0, \, 1]} \le
\overline{\theta}^{\, -1}-\sigma, \quad f\in N.
\end{align}
Suppose the contrary. Similarly as in \cite[p. 389]{vas_spn} it
can be proved that there exists a function $x$ such that
$\|x-P_Lx\|_{L_{q,v}[0, \, 1]} =\overline{\theta}^{\, -1}$,
$\left\| \frac{x^{(r)}}{g}\right\| _{L_p[0, \, 1]}=1$; in
addition, in case 1 the equality $\frac{\ddot{x}}{g} =\psi_{{\bf
c}}$ cannot hold for any ${\bf c}\in {\cal C}$ (see
(\ref{c_c_in_rm})), and in case 2 $x^{(r)}\ne \pm
\overline{x}^{(r)}$. Hence, there exists $\tau_*\in (0, \, 1)$
such that $\left. \frac{x^{(r)}} {\overline{x}^{(r)}} \right|
_{(\tau_*-\varepsilon, \, \tau_*+\varepsilon)}\ne {\rm const}$ for
any $\varepsilon>0$; moreover, $\tau_*\notin \{\eta_j:\; j\in
\Lambda\}$ in case 1. By Corollary \ref{r32l}, there exist
intervals $\Delta'$, $\Delta'' \subset (0, \, 1)$ such that $G(t,
\, \tau, \, \xi, \, \eta)\ne 0$ for a.e. $(t, \, \tau) \in
\Delta'\times \Delta''$ and $\left. \frac{x^{(r)}}
{\overline{x}^{(r)}} \right| _{\Delta''} \ne {\rm const}$. By
(\ref{ovr_x}) and (\ref{xtgt}),
$$
(-1)^{r-k} \frac{G(t, \, \tau, \, \xi, \, \eta)
\overline{x}^{(r)}(\tau)}{\overline{x}(t)} \ge 0, \quad \int
\limits _0^1 \frac{G(t, \, \tau, \, \xi, \, \eta)
\overline{x}^{(r)}(\tau)}{\overline{x}(t)}\, d\tau = (-1)^{r-k}.
$$
Arguing as in \cite[p. 389--390]{vas_spn}, we lead to a
contradiction. This completes the proof of (\ref{fplf}).

We construct a subspace $\tilde L$ of dimension $n+1$ such that
\begin{align}
\label{sup_x_wrpg} \sup _{x\in W^{r,k}_{p,g}[0, \, 1]} \inf _{z\in
\tilde L} \|x-z\| _{L_{q,v}[0, \, 1]} < \overline{\theta}^{\, -1}.
\end{align}
This implies (\ref{strict_monot}).

Consider case 1. Define the functions $\tilde \psi_j$ and the
space $\tilde L$ by formulas (\ref{til_psi_j}) and (\ref{tilde_l})
correspondingly. We choose $\delta>0$, $\gamma_0>0$, and for
arbitrary $\gamma\in (0, \, \gamma_0)$ we find $\gamma_1, \,
\dots, \, \gamma_{m+1}$ such that the following inequality holds
by Lemma \ref{r2_lem}:
\begin{align}
\label{dif_gamma} \int \limits_0^1 v^q(t)|Q_Lx(t)|^q\, dt-\int
\limits_0^1 v^q(t) |\tilde Qx(t)|^q\, dt\ge C_*\gamma, \quad x\in
M, \quad \gamma \in (0, \, \gamma_0).
\end{align}
For $x\in N$ we define the function $\tilde Px=\tilde P_\gamma x$
as follows: if $P_Lx=\sum \limits _{j=1}^n c_j(x)\psi_j$, then we
set $\tilde Px=\sum \limits _{j=1}^{l_m} c_j(x) \tilde \psi_j
+\sum \limits _{j=l_m+2}^{n+1}c_{j-1}(x)\tilde \psi_j$. By
assertion 2 of Lemma \ref{g_prop} and (\ref{hj_sum}), the set
$\{c_i(x):\; x\in W^{2,1}_{p,g}[0, \, 1], \;\; 1\le i\le n\}$ is
bounded; hence,
$$\sup _{x\in N} \|\tilde P_\gamma x -P_Lx\|_{L_{q,v}[0, \, 1]}
\underset{\gamma \to 0}{\to} 0.$$ Therefore, if $\gamma$ is
sufficiently small, then $\|x-\tilde Px\| _{L_{q,v}[0, \, 1]}
\stackrel{(\ref{fplf})}{\le} \overline{\theta}^{\,
-1}-\frac{\sigma}{2}$ for any $x\in N$. Since $x+\tilde Qx \in
\tilde L$ by Lemma \ref{r2_lem0}, this together with
(\ref{dif_gamma}) implies (\ref{sup_x_wrpg}).

In case 2 the arguments are similar as in \cite[p. 391]{vas_spn}.

\section{Applications}

In \cite{vas_alg_an} order estimates for Kolmogorov widths of
classes $W^{r,k}_{p,g}[0, \, e^{-1}]$ were obtained, where
\begin{align}
\label{g_v_def} g(x)=x^{-\beta_g}|\ln x|^{-\alpha_g}\rho_g(|\ln
x|), \quad v(x)=x^{-\beta_v}|\ln x|^{-\alpha_v}\rho_v(|\ln x|),
\end{align}
\begin{align}
\label{beta_g} \beta_g+\beta_v=r+\frac 1q-\frac 1p, \quad
\beta_v\notin \left\{ \frac 1q, \, \frac 1q+1, \, \dots, \, \frac
1q+r-1\right\},
\end{align}
\begin{align}
\label{alpha_g} \alpha_g+\alpha_v>\left(\frac 1q-\frac
1p\right)_+,
\end{align}
$\rho_g$, $\rho_v$ were absolutely continuous functions such that
\begin{align}
\label{rho_g} \lim \limits _{y\to +\infty}
\frac{y\rho_g'(y)}{\rho_g(y)}=\lim \limits _{y\to +\infty}
\frac{y\rho_v'(y)}{\rho_v(y)}=0.
\end{align}
Denote $\alpha=\alpha_g+\alpha_v$, $\rho(y)=\rho_g(y)\rho_v(y)$.

\begin{Trm}
\label{x_ln} Let $r\in \N$, $1<q\le p<\infty$, let conditions
(\ref{g_v_def}), (\ref{beta_g}), (\ref{alpha_g}), (\ref{rho_g})
hold, where $\beta_v \in \left(\frac 1q+k-1, \, \frac
1q+k\right)$, $1\le k\le r-1$. Let $\overline{\theta}_n=\sup
sp_n$, where the set $sp_n$ is defined according to
(\ref{eq_main}). Then
\begin{align}
\label{theta_n_est} \overline{\theta}_n \asymp \left\{
\begin{array}{l} n^r, \quad \alpha>r+\frac 1q-\frac 1p, \\
n^r(\log n)^{-r-\frac 1q+\frac 1p}, \quad \alpha>r+\frac 1q-\frac
1p, \quad \rho_g\equiv 1, \; \rho_v\equiv 1, \\ n^{\alpha-\frac
1q+\frac 1p}[\rho(n)]^{-1}, \quad \alpha<r+\frac 1q-\frac 1p.
\end{array} \right.
\end{align}
Moreover, if $\alpha$ and $\rho$ are such that $gv\in
L_\varkappa[0, \, e^{-1}]$ with $\frac{1}{\varkappa}=r +\frac
1q-\frac 1p$, then
\begin{align}
\label{th_asymp} \lim \limits _{n\to \infty} n^r \theta_n^{-1} =
\lambda_{rqp}^{-1}\|gv\|_{L_\varkappa [0, \, e^{-1}]},
\end{align}
where $\lambda_{rqp}$ is the first eigenvalue for the problem
$$
(-1)^{r+1}((x^{(r)})_{(p)})^{(r)}+ \lambda^q x_{(q)}=0
$$
with periodic boundary conditions.
\end{Trm}
\begin{proof}
By Theorem \ref{main_trm},
$$
\theta_n^{-1} = d_n (W^{r,k}_{p,g}[0, \, e^{-1}]).
$$
This together with \cite[Theorem 1, Examples 1--3]{vas_alg_an}
yields (\ref{theta_n_est}).

Let us prove (\ref{th_asymp}). To this end we check that
\begin{align}
\label{iabk_est} \|\tilde I_{r,g,v}^{a,b,k}\varphi\| _{L_q[a, \,
b]}\underset{g, \, v, \, p, \, q, \, r}{\lesssim}
\|gv\|_{L_\varkappa[a, \, b]}\|\varphi\|_{L_p[a, \, b]}.
\end{align}

If $a\ge \frac b2$, then (\ref{iabk_est}) follows from
\cite[Proposition 2]{vas_alg_an}. Let $a<\frac b2$. Then from
\cite[Propositions 3, 4]{vas_alg_an} it follows that
\begin{align}
\label{1t} \|\tilde I_{r,g,v}^{a,b,k}\varphi\| _{L_q[a, \,
b]}\underset{g, \, v, \, p, \, q, \, r}{\lesssim} \left(\int
\limits _{|\log b|}^{|\log a|} t^{-\frac{\alpha
pq}{p-q}}[\rho(t)]^{\frac{pq}{p-q}}\, dt\right)^{\frac 1q-\frac
1p}
\end{align}
(for $p=q$ we take the norm in $L_\infty$). On the other hand,
\begin{align}
\label{2t} \|gv\|_{L_\varkappa[a, \, b]}\|\varphi\|_{L_p[a, \, b]}
=\left(\int \limits _{|\log b|}^{|\log a|} t^{-\alpha \varkappa}
[\rho(t)]^\varkappa \, dt\right)^{\frac{1}{\varkappa}}.
\end{align}
From the inequality $\varkappa <\frac{pq}{p-q}$ it follows that
$\|(x_n)_{n\in \N}\|_{l_{\frac{pq}{p-q}}} \le \|(x_n)_{n\in
\N}\|_{l_\varkappa}$ for any sequence $(x_n)_{n\in \N}$. Writing
the integrals (\ref{1t}) and (\ref{2t}) in terms of sums of
integrals over intervals of length $l\in [1, \, 2]$, we obtain
(\ref{iabk_est}).

In order to complete the proof of (\ref{th_asymp}) we argue
similarly as in \cite[\S 4]{vas_m_sb}: we apply the Buslaev's
result \cite{busl_aa} for piecewise-continuous weights and pass to
limit.
\end{proof}

\end{document}